\documentclass[a4paper,leqno,12pt]{amsart}
\usepackage[all]{xy}
\usepackage{xspace}
\usepackage{amsmath}
\usepackage{amstext}
\usepackage{amsfonts}
\usepackage[mathscr]{euscript}
\usepackage{amscd}
\usepackage{latexsym}
\usepackage{amssymb}
\usepackage{diagrams}
\usepackage{enumerate}
\theoremstyle{plain}
\swapnumbers
    \newtheorem{thm}{Theorem}[section]
    \newtheorem{prop}[thm]{Proposition}
    \newtheorem{lemma}[thm]{Lemma}
    \newtheorem{corollary}[thm]{Corollary}
    
    \newtheorem{subsec}[thm]{}
    \newtheorem*{thma}{Theorem A}
    \newtheorem*{thmb}{Theorem B}
    \newtheorem*{thmc}{Theorem C}

\theoremstyle{definition}
    \newtheorem{defn}[thm]{Definition}
    \newtheorem{example}[thm]{Example}

    \newtheorem{notation}[thm]{Notation}

\theoremstyle{remark}
        \newtheorem{remark}[thm]{Remark}
        
        \newtheorem{warning}[thm]{Warning}
    
    \newtheorem{ack}[thm]{Acknowledgements}
\newenvironment{myeq}[1][]
{\stepcounter{thm}\begin{equation}\tag{\thethm}{#1}}
{\end{equation}}

\newcommand{\mydiagram}[2][]
{\stepcounter{thm}\begin{equation}
     \tag{\thethm}{#1}\vcenter{\xymatrix{#2}}\end{equation}}
\newcommand{\mysdiag}[2][]
{\stepcounter{thm}\begin{equation}
     \tag{\thethm}{#1}\vcenter{\xymatrix@R=20pt@C=15pt{#2}}\end{equation}}
\newenvironment{mysubsection}[2][]
{\begin{subsec}\begin{upshape}\begin{bfseries}{#2.}
\end{bfseries}{#1}}
{\end{upshape}\end{subsec}}
\newenvironment{mysubsect}[2][]
{\begin{subsec}\begin{upshape}\begin{bfseries}{#2\vsn.}
\end{bfseries}{#1}}
{\end{upshape}\end{subsec}}
\newcommand{\sect}{\setcounter{thm}{0}\section}
\newcommand{\wh}{\ -- \ }
\newcommand{\wwh}{-- \ }
\newcommand{\w}[2][ ]{\ \ensuremath{#2}{#1}\ }
\newcommand{\ww}[1]{\ \ensuremath{#1}}
\newcommand{\www}[2][ ]{\ensuremath{#2}{#1}\ }
\newcommand{\wwb}[1]{\ \ensuremath{(#1)}-}
\newcommand{\wb}[2][ ]{\ (\ensuremath{#2}){#1}\ }
\newcommand{\wref}[2][ ]{\ (\ref{#2}){#1}\ }
\newcommand{\hsp}{\hspace*{7 mm}}

\newcommand{\hsm}{\hspace*{2 mm}}
\newcommand{\vsn}{\vspace{2 mm}}

\newcommand{\vsm}{\vspace{4 mm}}
\newcommand{\hra}{\hookrightarrow}
\newcommand{\xra}[1]{\xrightarrow{#1}}
\newcommand{\xepic}[1]{\xrightarrow{#1}\hspace{-5 mm}\to}

\newcommand{\lra}[1]{\langle{#1}\rangle}

\newcommand{\epic}{\to\hspace{-3 mm}\to}
\newcommand{\xhra}[1]{\overset{#1}{\hookrightarrow}}
\newcommand{\efp}{\to\hspace{-1.5 mm}\rule{0.1mm}{2.2mm}\hspace{1.2mm}}
\newcommand{\efpic}{\mbox{$\to\hspace{-3.5 mm}\efp$}}
\newcommand{\hefpic}{\mbox{$\to\hspace{-2.5 mm}\to\hspace{-1.0 mm}\rule{0.1mm}{2.2mm}\hspace{1.2mm}$}}

\newcommand{\rest}[1]{\lvert\sb{#1}}
\newcommand{\Hu}[3]{H\sp{#1}({#2};{#3})}

\newcommand{\HuR}[1]{\Hu{\ast}{#1}{R}}
\newcommand{\ab}{\operatorname{ab}}

\newcommand{\chr}{\operatorname{char}}
\newcommand{\colim}{\operatorname{colim}}
\newcommand{\core}{\operatorname{core}}
\newcommand{\diag}{\operatorname{diag}}
\newcommand{\disc}{\operatorname{d}}

\newcommand{\ev}{\operatorname{ev}}
\newcommand{\gr}{\operatorname{gr}}
\newcommand{\ho}{\operatorname{ho}}

\newcommand{\Hom}{\operatorname{Hom}}
\newcommand{\hHom}{\widehat{\Hom}}

\newcommand{\Hm}[1]{{#1}\sp{\dag}}
\newcommand{\hvare}{\varepsilon\sp{\dag}}
\newcommand{\Kan}{\operatorname{Kan}}
\newcommand{\Ker}{\operatorname{Ker}}
\newcommand{\Id}{\operatorname{Id}}
\newcommand{\Image}{\operatorname{Im}}
\newcommand{\hIm}{\widehat{\Image}}
\newcommand{\hhIm}{\widehat{\widehat{\Image}}}
\newcommand{\inc}{\operatorname{inc}}
\newcommand{\New}{\operatorname{New}}
\newcommand{\Obj}{\operatorname{Obj}\,}
\newcommand{\op}{\sp{\operatorname{op}}}
\newcommand{\proj}{\operatorname{proj}}

\newcommand{\Stov}{\operatorname{St}}
\newcommand{\Tot}{\operatorname{Tot}}
\newcommand{\uTot}{\underline{\Tot}}
\newcommand{\map}{\operatorname{map}}

\newcommand{\mapa}{\map\sb{\ast}}

\newcommand{\A}{\mathbf{A}}
\newcommand{\eA}{{\EuScript A}}

\newcommand{\eB}{{\EuScript B}}
\newcommand{\C}{\mathcal{C}}

\newcommand{\F}{\mathcal{F}}
\newcommand{\hF}{\widehat{\F}}
\newcommand{\G}{\mathcal{G}}
\newcommand{\HR}[1]{{\EuScript H}{#1}}

\newcommand{\LA}{{\mathcal L}\sb{\F}}

\newcommand{\LG}{\hat{L}\sb{\G}}
\newcommand{\LGp}{\hat{L}\sb{\G'}}
\newcommand{\LR}[1]{\widehat{{\mathcal L}}\sb{R}\sp{#1}}
\newcommand{\M}{{\mathcal M}}
\newcommand{\UM}{M}
\newcommand{\UPM}{PM}

\newcommand{\sMR}{s\M\sb{R}}

\newcommand{\Map}{{\EuScript Map}}
\newcommand{\MA}{\Map\sb{\A}}
\newcommand{\MF}{\Map\sb{\TsF}}

\newcommand{\MT}{\Map\sb{\bT}}

\newcommand{\MFrd}{\Map\sb{\F,\disc}}

\newcommand{\MRrd}{\widehat{\Map}\sb{R,\disc}}
\newcommand{\Mop}{\MFrd\op}
\newcommand{\MRop}{\MRrd\op}
\newcommand{\MFS}{\Map\sb{\F}\sp{\Stov}}
\newcommand{\MAS}{\MFS}

\newcommand{\OO}{\mathcal{O}}
\newcommand{\PP}{\mathcal{P}}
\newcommand{\cT}{\mathcal T}
\newcommand{\RF}{{\mathcal R}\sb{\F}}
\newcommand{\RA}{\RF}

\newcommand{\RR}[1]{\widehat{{\mathcal R}}\sb{R}\sp{#1}}
\newcommand{\SF}{{\mathcal S}\sb{\F}}
\newcommand{\SA}{\SF}

\newcommand{\SR}[1]{\widehat{\mathcal S}\sb{R}\sp{#1}}
\newcommand{\TF}{\cT\sb{\F}}

\newcommand{\cTR}{{\EuScript T}\sb{R}}
\newcommand{\hTR}[1]{\widehat{\cT}\sb{R}\sp{#1}}
\newcommand{\ak}[1]{a\sp{#1}}
\newcommand{\dz}[1]{d\sp{0}\sb{#1}}
\newcommand{\od}{\overline{\partial}}
\newcommand{\ud}{\overline{d}}
\newcommand{\vdz}[1]{{\underline{d}\sp{0}\sb{#1}}}
\newcommand{\odz}[1]{\od\sp{#1}\sb{0}}
\newcommand{\udz}[1]{\ud\sp{0}\sb{#1}}
\newcommand{\td}{{\raisebox{-1.5ex}{$\stackrel{\textstyle d}{\sim}$}}}
\newcommand{\tdz}[1]{\td\quad\hspace*{-4mm}\sp{0}\sb{#1}}
\newcommand{\etk}[1]{\eta\sp{#1}}
\newcommand{\Fk}[1]{F\sp{#1}}
\newcommand{\oph}{\overline{\varphi}}
\newcommand{\ophi}[2]{\oph\sp{#1}\sb{[{#2}]}}
\newcommand{\prn}[1]{\pi\sb{[{#1}]}}
\newcommand{\prnk}[2]{\pi\sb{[{#1}]}\sp{#2}}
\newcommand{\psn}[2]{\psi\sp{#1}\sb{[{#2}]}}
\newcommand{\qk}[1]{q\sp{#1}}
\newcommand{\vare}{\varepsilon}
\newcommand{\vn}[1]{\vare\sb{[{#1}]}}
\newcommand{\svn}[1]{\overline{\vare\sb{[{#1}]}}}
\newcommand{\zn}[2]{\zeta\sp{#1}\sb{[{#2}]}}
\newcommand{\Set}{{\EuScript Set}}
\newcommand{\Seta}{\Set\sb{\ast}}
\newcommand{\Grp}{{\EuScript Gp}}
\newcommand{\Top}{{\EuScript Top}}
\newcommand{\Topa}{\Top\sb{\ast}}

\newcommand{\cS}{{\EuScript S}}
\newcommand{\SGa}{\Seta\sp{\Gamma}}
\newcommand{\Del}{\mathbf{\Delta}}
\newcommand{\Dp}{\Del\sp{+}}
\newcommand{\Du}{\Del\sp{\bullet}}
\newcommand{\Sa}{\cS\sb{\ast}}
\newcommand{\Sk}{\Sa\sp{\Kan}}
\newcommand{\bT}{\mathbf{\Theta}}
\newcommand{\TsA}{\bT\sb{\eA}}

\newcommand{\TsR}{\bT\sb{R}}
\newcommand{\ThA}{\Theta\sb{\eA}}
\newcommand{\TsF}{\bT\sb{\F}}
\newcommand{\TFs}[1]{\bT\sb{\F}\sp{\Stov,{#1}}}
\newcommand{\ThF}{\Theta\sb{\F}}

\newcommand{\TR}{\Theta\sb{R}}
\newcommand{\FF}{\mathbb F}
\newcommand{\Fp}{\FF\sb{p}}
\newcommand{\Fq}{\FF\sb{q}}
\newcommand{\NN}{\mathbb N}
\newcommand{\QQ}{\mathbb Q}
\newcommand{\ZZ}{\mathbb Z}
\newcommand{\bA}{{\mathbf A}}
\newcommand{\bB}{{\mathbf B}}
\newcommand{\uA}[1]{\underline{\bA}\sb{#1}}
\newcommand{\uB}[1]{\underline{\bB}\sb{#1}}
\newcommand{\bK}{{\mathbf K}}
\newcommand{\KP}[2]{\bK({#1},{#2})}
\newcommand{\KR}[1]{\KP{R}{#1}}
\newcommand{\KZ}[1]{\KP{\ZZ}{#1}}

\newcommand{\uL}{\underline{\Lambda}}
\newcommand{\bS}[1]{{\mathbf S}\sp{#1}}
\newcommand{\bU}{{\mathbf U}}

\newcommand{\bW}{{\mathbf W}}
\newcommand{\bX}{{\mathbf X}}
\newcommand{\bY}{{\mathbf Y}}

\newcommand{\bZ}{{\mathbf Z}}
\newcommand{\hQ}[1]{\widehat{Q}\sb{#1}}
\newcommand{\hdel}{\widehat{\delta}}
\newcommand{\het}[1]{\widehat{\eta}\sb{#1}}
\newcommand{\hmu}[1]{\widehat{\mu}\sb{#1}}
\newcommand{\home}[1]{\widehat{\omega}\sb{#1}}

\newcommand{\hvar}{\widehat{\vare}}
\newcommand{\hxi}[1]{\widehat{\xi}\sb{#1}}
\newcommand{\cu}[1]{c({#1})\sp{\bullet}}
\newcommand{\cd}[1]{c({#1})\sb{\bullet}}
\newcommand{\Ad}{A\sb{\bullet}}
\newcommand{\Bu}{B\sp{\bullet}}
\newcommand{\Gd}{G\sb{\bullet}}
\newcommand{\Gu}{G\sp{\bullet}}
\newcommand{\oG}[1]{\overline{G}\sb{#1}}

\newcommand{\Ud}{\bU\sb{\bullet}}

\newcommand{\Vd}{V\sb{\bullet}}
\newcommand{\Vu}{V\sp{\bullet}}
\newcommand{\oV}[1]{\overline{V}\sb{#1}}
\newcommand{\Wu}{\bW\sp{\bullet}}
\newcommand{\hW}[1]{\widehat{\bW}\sp{#1}}
\newcommand{\hWu}{\hW{\bullet}}
\newcommand{\hWl}[1]{\widehat{\bW}\sp{#1}\sb{\lambda}}
\newcommand{\hWlu}{\hWl{\bullet}}
\newcommand{\Wn}[2]{\bW\sp{#1}\sb{[#2]}}
\newcommand{\W}[1]{\Wn{\bullet}{#1}}
\newcommand{\tW}{{\raisebox{-1.5ex}{$\stackrel{\textstyle \bW}{\sim}$}}}
\newcommand{\tWn}[2]{\tW\quad\hspace*{-4mm}\sp{#1}\sb{[#2]}}

\newcommand{\uW}[1]{\overline{\bW}\sp{#1}}
\newcommand{\uWn}[2]{\uW{#1}\sb{[#2]}}
\newcommand{\vW}[1]{\widehat{\bW}\sp{#1}}
\newcommand{\vWn}[2]{\vW{#1}\sb{[#2]}}
\newcommand{\vWu}[1]{\vWn{\bullet}{#1}}
\newcommand{\Xu}{\bX\sp{\bullet}}
\newcommand{\Yu}{\bY\sp{\bullet}}
\newcommand{\Alg}[1]{{#1}\text{-}{\EuScript Alg}}
\newcommand{\Aal}[1][ ]{$\ThA$-algebra{#1}}
\newcommand{\Fal}[1][ ]{$\ThF$-algebra{#1}}

\newcommand{\Tal}[1][ ]{$\Theta$-algebra{#1}}
\newcommand{\TRal}[1][ ]{$\TR$-algebra{#1}}

\newcommand{\TAlg}{\Alg{\Theta}}

\newcommand{\FAlg}{\Alg{\ThF}}

\newcommand{\TRA}{\Alg{\TR}}

\newcommand{\Mod}[1]{{#1}\text{-}{\EuScript Mod}}
\newcommand{\ma}[1][ ]{mapping algebra{#1}}
\newcommand{\Ama}[1][ ]{$\TsA$-mapping algebra{#1}}
\newcommand{\dRma}[1][ ]{semi-discrete $R$-mapping algebra{#1}}
\newcommand{\Tma}[1][ ]{$\bT$-mapping algebra{#1}}
\newcommand{\Rma}[1][ ]{$\TsR$-mapping algebra{#1}}
\newcommand{\Sma}[1][ ]{$\F$-Stover mapping algebra{#1}}
\newcommand{\Fma}[1][ ]{$\TsF$-mapping algebra{#1}}
\newcommand{\dFma}[1][ ]{discrete $\F$-mapping algebra{#1}}

\newcommand{\lin}[1]{\{{#1}\}}
\newcommand{\fff}{\mathfrak{f}}
\newcommand{\fG}{\mathfrak{G}}
\newcommand{\fM}{\mathfrak{M}}
\newcommand{\fMT}{\fM\sb{\bT}}
\newcommand{\fMA}{\fM\sb{\eA}}

\newcommand{\fMR}{\fM\sb{R}}
\newcommand{\fMFS}{\fM\sb{\F}\sp{\Stov}}
\newcommand{\fMS}{\fMFS}
\newcommand{\fMF}{\fM\sb{\F}}
\newcommand{\fV}{\mathfrak{V}}
\newcommand{\fVn}[2]{\fV\sp{[{#1}]}\sb{#2}}
\newcommand{\fVd}{\fV\sb{\bullet}}
\newcommand{\fVt}[1]{\tilde{\fV}\sb{#1}}
\newcommand{\fVdd}{\fVt{\bullet}}
\newcommand{\fVnd}[1]{\fV\sp{[{#1}]}\sb{\bullet}}

\newcommand{\fW}{\mathfrak{W}}
\newcommand{\fWd}{\fW\sb{\bullet}}
\newcommand{\fX}{\mathfrak{X}}

\newcommand{\fY}{\mathfrak{Y}}
\newcommand{\fZ}{\mathfrak{Z}}
\newcommand{\bk}{[\mathbf{k}]}
\newcommand{\bn}{[\mathbf{n}]}
\begin{document}
%
%
\title{Mapping spaces and $R$-completion}
\author{David Blanc}
\author{Debasis Sen}
\address{Department of Mathematics\\ University of Haifa\\
31905 Haifa\\ Israel}
\email{blanc@math.haifa.ac.il,\ sen\_deba@math.haifa.ac.il}

\date{\today}


\begin{abstract}
We study the questions of how to recognize when a simplicial set $X$ is of the
form \w[,]{X=\mapa(\bY,\bA)} for a given space $\bA$, and how to recover $\bY$
from $X$, if so. A full answer is provided when \w[,]{\bA=\KR{n}}
for \w{R=\Fp} or $\QQ$, in terms of a \emph{\ma} structure on $X$ (defined in
terms of product-preserving simplicial functors out of a certain
simplicially-enriched sketch $\bT$). In addition, when
\w{\bA=\Omega\sp{\infty}\eA} for a suitable connective ring spectrum $\eA$,
we can \emph{recover} $\bY$ from \w[,]{\mapa(\bY,\bA)}
given such a \ma structure. Most importantly, our methods  provide a new
way of looking at the classical Bousfield-Kan $R$-completion.
\end{abstract}

\maketitle

\setcounter{section}{0}

%
%
\section*{Introduction}
\label{cint}

Given pointed topological spaces $\bA$ and $\bY$, one can form the
simplicial mapping space \w[,]{\mapa(\bY,\bA)} which models
the topological space \w{\Hom_{\Topa}(\bY,\bA)} of all pointed continuous maps
(with the compact-open topology). Such mapping spaces
play a central role in modern homotopy theory, so it is natural to ask
when a simplicial set $X$ is of the form \w[,]{\mapa(\bY,\bA)}
up to weak equivalence. We assume that one of the two spaces $\bA$ and
$\bY$ is given, but not both.

This problem has been extensively studied in the case where \w{\bY=\bS{n}}
is a sphere, so $X$ is an $n$-fold loop space (see, e.g.,
\cite{SugH,StaH,MayG,BadzA}). A general answer for any pointed $\bY$ was
provided in \cite{BBDoraR}.

Here we consider the dual problem: given a space $\bA$ and a
simplicial set $X$,

\begin{enumerate}
\renewcommand{\labelenumi}{(\alph{enumi})~}
\item when is $X$ of the form \w{\mapa(\bY,\bA)} for some space $\bY$?
\item If $X$ satisfies the conditions prescribed in the answer to (a),
  how can we recover $\bY$ from it?
\end{enumerate}

Observe that the best we can hope for is to recover $\bY$ up $\bA$-equivalence \wh
where a map \w{f:\bY\to\bY'} is called an $\bA$-\emph{equivalence} if it induces
a weak equivalence \w[.]{\mapa(\bY',\bA)\simeq\mapa(\bY,\bA)}

Unfortunately, the methods of \cite{BBDoraR} do not carry over in full to
the dual problem: we can answer both questions only when \w{\bA=\KR{n}} for
\w{R=\QQ} or \w[.]{\Fp} However, if we know that $X$ is a mapping space, and
only wish to recover $\bY$, we can do a little better: in this case we can
allow $\bA$ to be \w{\Omega\sp{\infty}\eA} for a suitable ring
spectrum $\eA$ (in particular, $\bA$ can be \w{\KR{n}} for any commutative
ring $R$).
Nevertheless, this is probably the most useful example of the problem
under consideration, since the most important representable functors studied in
homotopy theory are the generalized cohomology theories corresponding
to such $\Omega$-spectra.

\begin{remark}\label{realma}
Note that when \w[,]{\bA=\KP{V}{n}} \w{X:=\mapa(\bY,\bA)} is itself
a GEM \wh i.e., it is homotopy equivalent to a product of
Eilenberg-Mac~Lane spaces \wh so it appears to carry very little
homotopy-invariant information. Thus at first sight it seems unlikely that
one could recover anything but the $R$-cohomology of $\bY$ from $X$.

In particular, any GEM
\w{X\simeq\prod\sb{i=1}\sp{n}\,\KP{M\sb{i}}{i}} for finitely-generated
abelian groups \w{M\sb{0},\dotsc,M\sb{n}} is homotopy equivalent to
\w{\mapa(\bY,\KZ{n})} for some space $\bY$ \wh e.g., for
\w[,]{\bY\simeq\bigvee\sb{i=0}\sp{n}\,L(M\sb{i},i)} where \w{L(M\sb{i},i)}
is a co-Moore space (cf.\ \cite{KainC}). Thus the answer to (a) is always
positive, for such an $X$.

There are two difficulties with this point of view:
\begin{enumerate}
\renewcommand{\labelenumi}{(\roman{enumi})~}
\item This will not necessarily work for other rings $R$, or if the abelian
groups \w{M\sb{i}} are not finitely generated (see \cite{GGonCM} and
\S \ref{rallow} below).
\item There are usually many other possible choices of $\bY$, and
the homotopy type of $X$ alone does not enable us to distinguish between them.
\end{enumerate}

With reference to the second point, we must therefore impose sufficient
additional structure on \w{X=\mapa(\bY,\bA)} to allow us to recover $\bY$,
uniquely up to $\bA$-equivalence.  This structure is defined as follows:
\end{remark}

\begin{mysubsection}{Mapping algebras}
\label{smapalg}
In the original example of loop space recognition mentioned above, the
additional data on \w{\Omega\bY=\mapa(\bS{1},\bA)} consisted of the group
structure, induced by the cogroup structure map
\w[,]{\nabla:\bS{1}\to\bS{1}\vee\bS{1}} its iterates, and various (higher)
homotopies between them. These are all encoded in the actions, by composition,
of pointed mapping spaces between wedges of circles on \w{\Omega\bY} (and on
\w[,]{\Omega\bA\times\Omega\bA=\mapa(\bS{1}\vee\bS{1},\bA)} and so on).
When \w[,]{\bY=\bS{n}} this is described in terms of a suitable operad
(see \cite{MayG}), but a codification of the additional structure needed
for an arbitrary pointed space $\bY$ was provided in \cite{BBlaC,BBDoraR}.

The latter suggests that in the dual case one should look the action of the
mapping spaces \w{\mapa(\prod_{i=1}^{n}\bA,\,\bA)} on
\w[.]{\mapa(\bY,\prod_{i=1}^{n}\bA)=\prod_{i=1}^{n}\mapa(\bY,\bA)} Moreover,
since \w[,]{\Omega\mapa(\bY,\bA)=\mapa(\bY,\Omega\bA)} we should consider
also maps products of various loop spaces of $\bA$.

Although some of our results are valid for any $H$-group $\bA$,
their full force requires the ability to deloop $\bA$ arbitrarily,
so we start with an $\Omega$-spectrum \w[,]{\eA:=(\uA{n})\sb{n\in\ZZ}}
where we might have \w{\bA=\uA{0}} in mind for our original questions (a) and (b)
above. We can then formalize the additional structure mentioned above as follows:

Let \w{\TsA} be the sub-category of \w{\Topa} (or any other simplicial category
$\C$) whose objects are generated by the spaces \w{\uA{n}} \wb{n\in\ZZ} under
products of cardinality \www[,]{<\lambda} for a suitably chosen cardinal $\lambda$.
Thus the objects of \w{\TsA} have the form \w[,]{\prod\sb{i\in I}\uA{n\sb{i}}}
for some \w{(n\sb{i})\sb{i\in I}} in $\ZZ$. This will be called an
\emph{enriched sketch}, since it inherits a simplicial enrichment from $\C$.

For example, if \w{\eA=\HR{R}:=(\KR{n})\sb{n=0}\sp{\infty}} is the
$R$-Eilenberg-Mac~Lane spectrum for some ring $R$, and \w[,]{\lambda=\aleph\sb{0}}
then the objects of \w{\TsA} are all finite type free $R$-module GEMs.

A simplicial functor \w{\fX:\TsA\to\Sa} which preserved all loops and
products will be called a \emph{\Ama[.]} The main  example we have in mind
is a \emph{realizable}  \Ama[,] of the form \w{\bB\mapsto\mapa(\bY,\bB)} for
some fixed object $\bY$ (and all \w[,]{\bB\in\TsA} of course).

Note that if \w[,]{X=\mapa(\bY,\uA{k})} the realizable \Ama $\fX$
corresponding to $\bY$ has
\w[,]{\fX(\prod\sb{i\in I}\uA{n_{i}})=\prod\sb{i\in I}\Omega\sp{k-n_{i}}X}
so we can think of $\fX$ as additional structure on the simplicial set $X$,
including choices of \emph{deloopings} of $X$.
\end{mysubsection}

It turns out that this structure is precisely what is needed to recover $\bY$
from $X$, under suitable assumptions. In fact, we show:

\begin{thma}
Let \w{\eA=(\uA{n})\sb{n=0}\sp{\infty}} be an $\Omega$-spectrum
model of a connective ring spectrum, with \w{R=\pi_{0}\eA} a commutative
ring, and let $\fX$ be a \Ama structure on
\w{X=\mapa(\bY,\uA{0})} for a simply-connected space $\bY$. We then can
construct functorially from $\fX$ a cosimplicial space \w{\Wu} with total space
\w{\Tot\Wu} weakly equivalent to the $R$-completion of $\bY$, and
\w[.]{X\simeq\mapa(\Tot\Wu,\uA{0})}
\end{thma}
\noindent See Theorem \ref{tresm}, Corollary \ref{ccommon}, and
Theorem \ref{trecogms} below.

\begin{mysubsection}{$\Theta$-algebras}
\label{sthetaalg}
If \w{\TsA} is an enriched sketch as above, applying \w{\pi\sb{0}} to each
of its mapping spaces yields a category \w[,]{\ThA:=\pi\sb{0}\TsA} which
is an ordinary (algebraic) \emph{sketch} in the sense of Ehresmann (see
\S \ref{dgalg}). If \w{\fX:\TsA\to\Sa} is a \Ama[,] composing with \w{\pi_{0}}
yields an \emph{\Aal} \w{\Lambda:=\pi\sb{0}\fX} \wwh that is, a product-preserving
functor \w[.]{\Lambda:\ThA\to\Seta} We can think of such an \Aal as an
algebraic version of a \Ama[.] Note that the category of simplicial
\Aal[s] has a model category structure as in \cite[II, \S 4]{QuiH},
which allows us to define free simplicial resolutions of $\Lambda$
(\S \ref{dcwres}).

For example, if \w{\eA=\HR{\Fp}}
(and \w[),]{\lambda=\aleph\sb{0}} \w{\ThA} is the homotopy category
of finite type \ww{\Fp}-GEMs, and a \Aal is just an algebra over the
mod $p$ Steenrod algebra, as in \cite[\S 1.4]{SchwU}.

Our main technical tool in this paper, which we hope will be of independent
use, is the following:
\end{mysubsection}

\begin{thmb}
If $\fX$ is a \Ama and \w{\Lambda=\pi_{0}\fX} is the corresponding \Aal[,]
then any algebraic CW-simplicial resolution \w{\Vd\to\Lambda}
can be realized by a cosimplicial space \w[.]{\Wu}
\end{thmb}
\noindent See Theorem \ref{tres} below.

\begin{warning}
We do not claim that the space \w{\bY:=\Tot\Wu} obtained by means of Theorem B in
fact realizes the \Ama $\fX$. In particular, if $\fX$ is an
\Ama structure on a simplicial set $X$, it need not be true that
\w[,]{X\simeq\mapa(\bY,\bA)} even for \w{\bA=\KR{n}}
(see \S \ref{rallow} below).
\end{warning}

However, it turns out that, for suitable fields $R$, the only obstruction to
realizability is a purely algebraic condition on
\w[.]{\Lambda=\pi\sb{0}\fX}  In particular, we show:

\begin{thmc}
When \w{\eA=\HR{R}} for \w{R=\QQ} or \w[,]{\Fp}
any simply-connected finite type \Ama $\fX$ is realizable by an
$R$-complete space $\bY$, unique up to $R$-equivalence.
\end{thmc}
\noindent See Theorem \ref{treal} below (which is somewhat more general)\vsm .

Theorem C implies that for \w{R=\QQ} or \w[,]{\Fp} a finite type
$R$-GEM $X$ which can be endowed with a simply-connected \Ama structure
(also involving deloopings of $X$) is realizable as a mapping space
\w{X\simeq\mapa(\bY,\KR{n})} for some $R$-complete $\bY$. Moreover, $\bY$ is
uniquely determined up to $R$-equivalence by the choice of \Ama structure.

\begin{mysubsection}{A new look at $R$-completion}
\label{snewpcompl}
The three theorems above, taken together, provide us with a new way of looking
at the concept of $R$-completion: more precisely, we have three notions
associated to any commutative ring $R$:
\begin{enumerate}
\renewcommand{\labelenumi}{(\alph{enumi})}
\item An \Ama $\fX$, for \w[,]{\eA=\HR{R}} with the accompanying algebraic
notion of the \Aal \w[;]{\Lambda=\pi\sb{0}\fX}
\item A cosimplicial $R$-resolution \w{\Wu} of the \Ama $\fX$, with the
accompanying algebraic notion of a free simplicial \Aal resolution
\w{\Vd} of \w{\pi\sb{0}\fX} (see \S \ref{sgrmc});
\item The realization \w{\Tot\Wu} of the $R$-resolution \w[,]{\Wu} which
is $R$-complete in the simply-connected case.
\end{enumerate}

When \w{R=\Fp} or $\QQ$, under mild assumptions on $\fX$ (e.g., when it is
simply connected and of finite type) we show that these notions are
equivalent, inasmuch as the original $\fX$ is in fact the \Ama of
\w[,]{\Tot\Wu} up to weak equivalence.

Moreover, each of these three notions has certain advantages:
\begin{enumerate}
\renewcommand{\labelenumi}{(\alph{enumi})}
\item The \Ama $\fX$ exhibits in an explicit form all the information about
a space $\bY$ which is retained by its $R$-completion.
\item The cosimplicial space \w[,]{\Wu} realizing the algebraic
resolution \w{\Vd} of \w[,]{H\sp{\ast}(\bY;R)} encodes the higher
order cohomology operations for $\bY$ in a visible manner
(see \S \ref{scrho} below)
\item The $R$-complete space \w{\Tot\Wu=R\sb{\infty}\bY} allows us to work with
a single space which retains all the above information.
\end{enumerate}
\end{mysubsection}

\begin{notation}\label{snac}
The category of topological spaces will be denoted by \w[,]{\Top} and that
of pointed topological spaces by \w[.]{\Topa}
The category of simplicial sets will be denoted by
\w[,]{\cS=s\Set} that of pointed simplicial sets by \w[,]{\Sa=s\Seta}
and that of pointed Kan complexes by \w[.]{\Sk}

Unless otherwise stated, $\C$ will be a pointed simplicial
model category in which every object is cofibrant (cf.\ \cite[II, \S 2]{QuiH}) \wh
so in particular it is left proper. We assume moreover that both $\otimes$
and $\times$ preserve cofibrations. The objects of $\C$ will be denoted by
boldface letters: \w[.]{\bX,\bY,\dotsc} The main example we shall be
concerned with is \w[,]{\Sa} and by \emph{space} we always mean
a (pointed) simplicial set.
\end{notation}

\begin{mysubsection}{Organization}\label{sorg}
In Section \ref{crmccs} we recall some facts about (co)simplicial objects,
and in particular Bousfield's resolution model category structure.
In Section \ref{ctma} we define the notions of enriched sketches and \ma[s,]
and in Section \ref{cmacr} we explain how this structure can be used to recover
$\bY$ from \w{X=\mapa(\bY,\bA)} by means of a suitable cosimplicial resolution
\w[,]{\bY\to\Wu} when \w{\bA=\KR{n}} for a field $R$.
In Section \ref{cmac} we modify the construction of this cosimplicial resolution
to obtain Theorem A.
In Section \ref{crstr} we prove our main technical result, Theorem B, showing
how any algebraic resolution of \w{\pi\sb{0}\fX} can be realized in \w[.]{c\C}
This is used in Section \ref{crma} to prove Theorem C, which allows us to
recognize mapping spaces of the form \w{\mapa(\bY,\KR{n})} and recover
$\bY$ up to $R$-completion.
\end{mysubsection}

\begin{ack}
We wish to thank Pete Bousfield for many useful comments and elucidations
of his work, and Paul Goerss for a helpful pointer.
\end{ack}

%
%
\sect{The Resolution Model Category of Cosimplicial Objects}
\label{crmccs}

The main technical tool for reconstructing the source $\bY$ of a
mapping space \w{X:=\map_{\C}(\bY,\bA)} in a pointed simplicial model
category $\C$, such as \w[,]{\Sa} is the construction of a suitable
cosimplicial resolution \w{\Wu} of the putative \w[.]{\bY\in\C}

The proper framework for obtaining such a \w{\Wu} is Bousfield's resolution
model category of cosimplicial objects over $\C$, which generalizes and
dualizes the Dwyer-Kan-Stover theory of the $E\sp{2}$-model category of simplicial
spaces (cf.\ \cite{DKStE}).

\begin{mysubsection}{Simplicial and cosimplicial objects}
\label{sscso}
We first collect some standard facts and constructions related to
(co)simplicial objects in any category $\C$:

Let $\Del$ denote the category of finite ordered sets and order-preserving
maps (cf.\ \cite[\S 2]{MayS}), and \w{\Dp} the subcategory with the same objects,
but only monic maps. A \emph{cosimplicial object} \w{\Gu} in a
category $\C$ is a functor \w[,]{\Del\to\C} and a \emph{restricted}
cosimplicial object is a functor \w[.]{\Dp\to\C} More concretely,  we write
\w{G\sp{n}} for the value of \w{\Gu} at the ordered set \w[.]{\bn=(0<1<\dotsc<n)}
The maps in the diagram \w{\Gu} are generated by the \emph{coface} maps
\w{d\sp{i}=d\sp{i}\sb{n}:G\sp{n}\to G\sp{n+1}} \wb[,]{0\leq i \leq n+1} as well
as \emph{codegeneracy} maps \w{s\sp{j}=s\sp{j}\sb{n}:G\sp{n}\to G\sp{n-1}}
\wb{0 \leq j<n} in the non-restricted case, satisfying the usual cosimplicial
identities.

Dually, a \emph{simplicial object} \w{\Gd} in $\C$ is a functor
\w[.]{\Del\op\to\C} The category \w{\C\sp{\Del}} of cosimplicial objects over
$\C$ will be denoted by \w[,]{c\C} and the category \w{\C\sp{\Del\op}} of
simplicial objects over $\C$ will be denoted by \w[.]{s\C}

There are natural embeddings \w{\cu{-}:\C\to c\C} and \w[,]{\cd{-}:\C\to s\C}
defined by letting \w{\cu{A}} denote the constant cosimplicial object
which is $A$ in every cosimplicial dimension, and similarly for \w[.]{\cd{A}}
\end{mysubsection}

\begin{mysubsection}{Latching and matching objects}
\label{slmo}
For a cosimplicial object \w{\Gu\in c\C} in a complete category $\C$,
the $n$-th \emph{matching object} for \w{\Gu} is defined to be
\begin{myeq}\label{eqmatch}
M\sp{n}\Gu~:=~\lim\sb{\phi:\bn\to \bk}\,G\sp{k}~,
\end{myeq}
\noindent where $\phi$ ranges over the surjective maps \w{\bn\to\bk} in
$\Del$. There is a natural map \w{\zeta\sp{n}:G\sp{n}\to M\sp{n}\Gu} induced
by the structure maps of the limit, and any iterated codegeneracy map
\w{s\sp{I}=\phi\sb{\ast}:G\sp{k}\to G\sp{n}} factors as
\begin{myeq}\label{equnivcod}
s\sp{I}~=~\proj\sb{\phi}\circ\zeta\sp{n}~,
\end{myeq}
\noindent where \w{\proj\sb{\phi}:M\sp{n}\Gu\to G\sp{k}} is the structure
map for the copy of \w{G\sp{k}} indexed by $\phi$ (cf.\ \cite[X,\S 4.5]{BKanH})\vsn.

Similarly, the $n$-th \emph{latching object} for \w{\Gu\in c\C} is the colimit
\begin{myeq}\label{eqlatch}
L\sp{n}\Gu~:=~\colim\sb{\theta:\bk\to\bn}\,G\sp{k}~,
\end{myeq}
\noindent where $\theta$ ranges over the injective maps \w{\bk\to\bn} in
$\Del$ (for \w[),]{k < n} with \w{\sigma\sp{n}:L\sp{n}\Gu\to G\sp{n}} defined by
the structure maps of the colimit.

These two constructions have analogues for a simplicial object \w{\Gd} over a
cocomp\-lete category $\C$: the \emph{latching object}
\w[,]{L\sb{n}\Gd:=\colim\sb{\theta:\bk\to\bn}G\sb{k}} and the
\emph{matching object} \w[,]{M\sb{n}\Gd:=\lim\sb{\phi:\bn\to\bk} G\sb{k}}
equipped with the obvious canonical maps.
\end{mysubsection}

\begin{defn}\label{dmco}
If $\C$ is pointed and complete, the $n$-th \emph{Moore chains} object
of \w{\Gd\in s\C} is defined to be:
\begin{myeq}\label{eqmoor}
C\sb{n}\Gd~:=~\cap\sb{i=1}\sp{n}\Ker\{d\sb{i}:G\sb{n}\to G\sb{n-1}\}~,
\end{myeq}
\noindent with differential \w[.]{\partial\sb{n}\sp{\Gd}=\partial\sb{n}:=
(d\sb{0})\rest{C\sb{n}\Gd}:C\sb{n}\Gd\to C\sb{n-1}\Gd}
The $n$-th \emph{Moore cycles} object is
\w[.]{Z\sb{n}\Gd:=\Ker(\partial\sb{n}\sp{\Gd})}

Dually, if $\C$ is pointed and cocomplete, the $n$-th \emph{Moore cochains}
object of \w[,]{\Gu\in s\C} written \w[,]{C\sp{n}\Gu} is defined to be
the colimit of:
\mydiagram[\label{eqmoorecc}]{
\coprod\sb{i=1}\sp{n}\,G\sp{n-1} \ar[d] \ar[rr]\sp{\bot\sb{i}\,d\sp{i}} &&
G\sp{n}\\
\ast &&
}
\noindent with differential \w{\delta\sp{n}:C\sp{n-1}\Gu\to C\sp{n}\Gu}
induced by \w[.]{d\sp{0}}
\end{defn}

\begin{defn}\label{dscwo}
A simplicial object \w{\Gd\in s\C} over a pointed category $\C$ is called a
\emph{CW object} if it is equipped with a \emph{CW basis}
\w{(\oG{n})\sb{n=0}\sp{\infty}} in $\C$ such that
\w[,]{G\sb{n}=\oG{n}\amalg L\sb{n}\Gd} and \w{d\sb{i}\rest{\oG{n}}=0}
for \w[.]{1\leq i\leq n} In this case
\w{\odz{G\sb{n}}:=d\sb{0}\rest{\oG{n}}:\oG{n}\to G\sb{n-1}} is called the
attaching map for \w[.]{\oG{n}} By the simplicial identities \w{\odz{G\sb{n}}}
factors as
\begin{myeq}\label{eqattach}
\odz{G\sb{n}}:\oG{n}~\to~Z\sb{n-1}\Gd~\subset~G\sb{n-1}~.
\end{myeq}

In this case we have an explicit description
\begin{myeq}\label{eqslatch}
L\sb{n}\Gd~:=~
\coprod\sb{0\leq k\leq n}~\coprod\sb{0\leq i\sb{1}<\dotsc<i\sb{n-k-1}\leq n-1}~
\oG{k}
\end{myeq}
\noindent for its $n$-th latching object, in which the iterated degeneracy map
\w[,]{s\sb{i\sb{n-k-1}}\dotsc s\sb{i\sb{2}}s\sb{i\sb{1}}} restricted to the
basis \w[,]{\oG{k}} is the inclusion into the copy of \w{\oG{k}} indexed by
\w[.]{(i\sb{1},\dotsc,i\sb{n-k-1})}
\end{defn}

A cosimplicial CW object may be defined analogously, but we shall only need
the following variant:

\begin{defn}\label{dccwo}
A cosimplicial pointed space \w{\Wu\in c\Sa} equipped with
a \emph{CW basis} \w{\uW{n}} \wb{n\geq 0} in \w{\Sa} is called a
\emph{weak CW object} if
\begin{enumerate}
\renewcommand{\labelenumi}{(\alph{enumi})}
\item For each \w[,]{n\geq 0} we have a weak equivalence
\w[,]{\varphi\sp{n}:\bW\sp{n}\xra{\simeq}\uW{n}\times M\sp{n}\Wu} and we set
\begin{myeq}\label{eqcwstru}
\oph\sp{n}~:=~\proj\sb{\uW{n}}\circ\varphi\sp{n}:\bW\sp{n}~\to~\uW{n}~.
\end{myeq}
\noindent where \w{\proj\sb{\uW{n}}:\bW\sp{n}\to\uW{n}} is the projection.
\item \w{\oph\sp{n}\circ d\sp{i}\sb{n-1}\sim 0} for \w[.]{1\leq i\leq n}
\item If we define the \emph{attaching map} for \w{\uW{n}} to be
\w[,]{\udz{n-1}:=\oph\sp{n-1}\circ d\sp{0}\sb{n-1}:\bW\sp{n-1}\to\uW{n}}
we require that it be a ``Moore cochain'' in the sense that
\begin{myeq}\label{eqvantwo}
\udz{n-1}\circ d\sp{i}\sb{n-2}~=~0
\end{myeq}
\noindent for all \w[.]{1\leq i\leq n-2}
\end{enumerate}
\end{defn}

\begin{remark}\label{rsmcat}
Recall that a \emph{simplicial} model category $\C$ is one in which,
for each (finite) \w{K\in\cS} and \w[,]{X\in\C} we have objects
\w{X\otimes K} and \w{X\sp{K}} in $\C$ equipped with appropriate
adjunction-like isomorphisms and axiom SM7 (see \cite[II, \S 1-2]{QuiH}).
In particular, such model categories are simplicially enriched.
\end{remark}

\begin{mysubsection}{Reedy model structure}
\label{sremc}
If $\C$ is a model category, the Reedy model structure on \w{c\C}
(cf.\ \cite[\S 2.2]{BousC}) is defined by letting a cosimplicial map
\w{f:\Xu\to \Yu} in \w{c\C} be:
\begin{enumerate}[(i)]
\item a \emph{Reedy weak equivalence} when \w{f\colon X^n \to Y^n} is a
weak equivalence in $\C$ for $n\geq 0$;
\item a \emph{Reedy cofibration} when \w{X^n\coprod\sb{L^n\Xu}L^n\Yu\to Y^n}
is a cofibration in $\C$ for $n\geq 0$;
\item  a \emph{Reedy fibration} when \w{X^n\to Y^n\prod\sb{M^n\Xu}M^n\Xu}
is a fibration in $\C$ for $n\geq 0$.
\end{enumerate}

The Reedy model category structure on \w{s\C} is defined dually (see
\cite[\S 15.3]{PHirM}).
\end{mysubsection}

\begin{defn}\label{dtotal}
The \emph{total space} \w{\Tot\Wu} of a cosimplicial
space \w{\Wu\in c\cS} is defined to be the simplicial mapping space
\w[,]{\map_{c\cS}(\Du,\Wu)} where \w{\Du\in c\cS} has
\w{\Del\sp{n}:=\Delta[n]\in\cS} (the standard $n$-simplex) \wh see
\cite[X, \S 3]{BKanH}.

We write \w[,]{\uTot\Wu:=\Tot\Vu} where \w{\Wu\to\Vu} is a functorial Reedy
fibrant replacement.
\end{defn}

\begin{mysubsection}{$\G$-resolution model structure}
\label{sgrmc}
Let $\G$ be a class of homotopy group objects in a pointed model
category $\C$, closed under loops. We shall be mainly interested in the case
where $\G$ consists of the \ww{\Omega^{\infty}}-spaces in a class of
$\Omega$-spectra in \w[.]{\C=\Sa}

A map \w{i:A\to B} in \w{\ho\C} is
called $\G$-\emph{monic} if \w{i\sp{\ast}:[B, G]\to [A,G]} is onto for
each \w[.]{G \in \G} An object $Y$ in $\C$ is called \emph{$\G$-injective} if
\w{i\sp{\ast}:[B, Y]\to [A,Y]} is onto for each $\G$-monic map \w{i:A \to B}
in \w[.]{\ho\C} A fibration in $\C$ is called \emph{$\G$-injective} if
it has the right lifting property for the $\G$-monic cofibrations in $\C$.

The homotopy category \w{\ho\C} is said to have \emph{enough $\G$-injectives}
if each object is the source of a $\G$-monic map to a $\G$-injective target.
In this case $\G$ is called a class of \emph{injective models} in \w[.]{\ho\C}

Recall that a homomorphism in the category \w{s\Grp} of simplicial groups
is a weak equivalence or fibration when its underlying map in $\cS$ is such.
A map \w{f:\Xu\to \Yu} in \w{c\C} is called
\begin{enumerate}[(i)]
 \item a $\G$-\emph{equivalence} if \w{f\sp{\ast}:[\Yu,G]\to[\Xu,G]}
is a weak equivalence in \w{s\Grp} for each \w[;]{G\in \G}
\item a $\G$-\emph{cofibration} if $f$ is a Reedy cofibration and
\w{f\sp{\ast}:[\Yu,G]\to [\Xu,G]} is a fibration in \w{s\Grp} for each
\w[;]{G \in \G}
\item a $\G$-\emph{fibration} if
\w{f:X\sp{n}\to Y\sp{n}\times\sb{M\sp{n}\Yu} M\sp{n}\Xu} is a
$\G$-injective fibration in $\C$ for each \w[.]{n\geq 0}
\end{enumerate}

In \cite[Theorem 3.3]{BousC}, Bousfield showed that if $\C$ is a left
proper pointed model category and $\G$ is a class of injective models in
\w[,]{\ho(\C)} the above defines a left proper pointed simplicial model
category structure on \w[.]{c\C}
\end{mysubsection}

\begin{defn}\label{dwgf}
Given a class $\G$ of homotopy group objects in a model category $\C$ as above,
a cosimplicial object \w{\Wu\in c\C} is called \emph{weakly $\G$-fibrant} if
it is Reedy fibrant, and every \w{W\sp{n}} is in $\G$ \wb[.]{n\geq 0}
A \emph{weak $\G$-resolution} of an object \w{Y\in\C} is a weakly
$\G$-fibrant \w{\Wu} which is $\G$-equivalent to \w{\cu{Y}}
(cf.\ \S \ref{sscso}). See \cite[\S 6]{BousC}.
\end{defn}

%
%
\sect{Enriched sketches and mapping algebras}
\label{ctma}

We now set up the categorical framework needed to describe the relevant
extra structure on a mapping space.

\begin{defn}\label{dgalg}
Let $\Theta$ be an \emph{sketch}, in the sense of
Ehresmann (cf.\ \cite{EhreET}, \cite[\S 5.6]{BorcH2}): that
is, a small pointed category with a distinguished set $\PP$ of (small)
products (including the empty product $\ast$).  A \emph{\Tal} is a functor
\w{\Lambda:\Theta\to\Seta} which preserves the products in $\PP$. We think of
a map \w{\phi:\prod\sb{i<\kappa}\,a\sb{i}\to\prod\sb{j<\lambda}\,b\sb{j}}
in $\Theta$ as representing an $\lambda$-valued $\kappa$-ary operation
on \Tal[s,] with gradings indexed by \w{(a\sb{i})\sb{i<\kappa}} and
\w[,]{(b\sb{j})\sb{j<\lambda}} respectively.

The category of \Tal[s] is denoted by \w[.]{\TAlg}
If each object of $\Theta$ is uniquely representable (up to order) as a
product of elements in a set \w[,]{\OO\subseteq\Obj\Theta} there is a forgetful
functor \w{U:\TAlg\to\Set\sp{\OO}} into the category of $\OO$-graded sets,
with left adjoint the \emph{free \Tal} functor \w[.]{F:\Set\sp{\OO}\to\TAlg}
\end{defn}

\begin{example}\label{eggalg}
The simplest kind of a sketch is a \emph{theory} in the sense of Lawvere
(cf.\ \cite{LawvF}), in which \w{\Obj\Theta=\NN} is generated under products
by a single object, so that \Tal[s] are simply sets with additional structure.
For example, the theory $\fG$ whose algebras are groups is just the opposite
category of the homotopy category of finite wedges of circles.
\end{example}

\begin{defn}\label{dfgsketch}
We define a $\fG$-\emph{sketch} to be a sketch $\Theta$
equipped with an embedding of sketches \w[,]{\fG\sp{\OO}\hra\Theta} for
$\OO$ as above. In this case, any \Tal $\Lambda$ has a
natural underlying $\OO$-graded group structure. We do not
require the operations of a $\fG$-sketch to be homomorphisms (that
is, commute with the $\fG$-structure).
\end{defn}

\begin{example}\label{eggth}
Almost all varieties of (graded) universal algebras, in the sense of
\cite[V, \S 6]{MacLC} \wh such as groups, associative, or Lie algebras, and
so on \wh have an underlying (graded) group structure, so they are
categories of \Tal[s] for a suitable $\fG$-sketch $\Theta$.
\end{example}

\begin{prop}\label{psimptal}
If $\Theta$ is a $\fG$-sketch, the category \w{s\TAlg} of simplicial
\Tal[s] has a model category structure, in which the weak equivalences
and fibrations are defined objectwise.
\end{prop}

\begin{proof}
See \cite[\S 6]{BPescF}, which is a slight generalization of \cite[II,\S 4]{QuiH}.
\end{proof}

\begin{mysubsection}{Enriched sketches and algebras}
There is also a enriched version of the notions defined above, introduced
in \cite{BBlaC}, in which we assume that the theory, or sketch, is simplicially
enriched, and the algebras over it are simplicial.  This takes place in the
context of a simplicially enriched model category. Note that in fact
any model category $\C$ can be enriched over $\cS$ \wh that is, for each
\w[,]{X,Y\in\C} there is a simplicial mapping space \w[,]{\map\sb{\C}(X,Y)}
with continuous compositions, such that \w{[X,Y]\sb{\ho\C}} is equal to its
set of components \w{\pi\sb{0}\map\sb{\C}(X,Y)} (cf.\ \cite{DKanF}).
\end{mysubsection}

\begin{defn}\label{deth}
Let $\C$ be a simplicial model category as in \S \ref{snac}, and $\lambda$ some
limit cardinal (to be determined by the context \wh see Remark \ref{rlambda}
below). An \emph{enriched sketch} $\bT$ is a small full sub-simplicial category
of $\C$, closed under loops (cf.\ \cite[I, \S 2]{QuiH}), with a distinguished set
of products $\PP$. We assume all objects in $\bT$ are fibrant
(and cofibrant) homotopy group objects in $\C$.
\end{defn}

\begin{example}\label{egeth}
Let \w{\F=\{\eA\sp{i}\}\sb{i\in I}} be a set of $\Omega$-spectra, so each
\w{\uA{n}\sp{i}\cong\Omega\uA{n+1}\sp{i}} is an \ww{\Omega\sp{\infty}}-space, and
let \w{\TsF} denote the full sub-simplicial category of $\C$ whose objects
are products of the spaces \w{\uA{n}\sp{i}} of cardinality \www[.]{<\lambda}
We write
\begin{myeq}\label{eqhatf}
\hF~:=~\{\uA{n}\sp{i}~:\ \eA\sp{i}\in\F,\,n\in\ZZ\}~\subseteq~\Obj\TsF~.
\end{myeq}

When \w{\F=\{\eA\}} consists of a single spectrum, we write \w{\TsA} for
\w[.]{\TsF} In particular, for any ring (or abelian group) $R$ we let
\w{\eA=\HR{R}} be the corresponding Eilenberg-Mac~Lane
spectrum, and denote the resulting enriched sketch \w{\TsA} (for
\w[)]{\lambda=\aleph\sb{0}} by \w[.]{\TsR} Thus the objects of \w{\TsR}
are finite type $R$-GEMs (generalized Eilenberg-Mac~Lane spaces) \wh that is,
spaces of the form \w{\prod\sb{i=1}^N\,\KR{m_i}} \wb[.]{m\sb{i}\geq 1}

Note that in this case we may assume that each object in \w{\TsR} is a strict
abelian group object.
\end{example}

\begin{remark}\label{rlambdaone}
We can always assume that the cardinal $\lambda$ for \w{\TsF} is $1$, by including
all products of $\Omega$-spectra of the requisite cardinality in the original list
$\F$ itself.  The reason for specifying $\lambda$ in the definition of an enriched
sketch is because we want to think of elements in
\w{\map\sb{\bT}(\prod\sb{i=0}\sp{n}\bB\sb{i},\bB)} as continuous $n$-ary
operations (as in the discrete case \wh cf.\ \S \ref{dgalg}). If we set
\w[,]{\lambda:=1} we must list the set $\PP$ of distinguished products in \w{\TsF}
explicitly.
\end{remark}

\begin{defn}\label{dpathloop}
In a pointed simplicial model category $\C$ (cf.\ \S \ref{rsmcat}), the
inclusions \w{i\sb{0},i\sb{1}:\ast\hra \Delta[1]} induce natural
``evaluation maps'' \w[,]{ev\sb{0},\ev\sb{1}:\bX\sp{\Delta[1]}\epic\bX}
which are trivial fibrations, for any \w[.]{\bX\in\C}  This allows one to
define the \emph{path} and \emph{loop} objects in $\C$ by the pullback diagrams:
\mydiagram[\label{eqpathloop}]{
\ar @{} [drr] |<<<<<{\framebox{\scriptsize{PB}}}
P\bX \ar@{^{(}->}[rr] \ar@{->>}[d]\sb{\simeq} &&
\bX\sp{\Delta[1]} \ar@{->>}[d]\sp{\simeq}\sb{\ev\sb{0}}&&
\ar @{} [drr] |<<<<<{\framebox{\scriptsize{PB}}}
\Omega\bX \ar@{^{(}->}[rr] \ar@{->>}[d] && P\bX \ar@{->>}[d]\sb{\ev\sb{1}} \\
\ast \ar@{^{(}->}[rr] && \bX && \ast \ar@{^{(}->}[rr] && \bX
}

These will also be our models for simplicial path and loop spaces \w{PK}
and \w{\Omega K} for any Kan complex \w[.]{K\in\Sk}
\end{defn}

\begin{defn}\label{dma}
For $\C$ a model category as in \S \ref{snac}, and \w{\bT\subseteq\C} an
enriched sketch as above, a \emph{\Tma} is a pointed simplicial functor
\w[,]{\fX:\bT\to\Sa} written \w{\fX:\bB\mapsto\fX\lin{\bB}} for any
\w[,]{\bB\in\bT} taking values in Kan complexes, and satisfying the
following three conditions:
\begin{enumerate}
\renewcommand{\labelenumi}{(\alph{enumi})~}
\item The natural map \w{\fX\lin{\prod\sb{i<\lambda}\,\bB\sb{i}}\to
\prod\sb{i<\lambda}\,\fX\lin{\bB\sb{i}}} is an isomorphism for products
\w{\prod\sb{i<\lambda}\bB\sb{i}} in our distinguished set $\PP$.
\item Using the convention that \w{\fX\lin{\bB\sp{K}}:=(\fX\lin{\bB})\sp{K}}
for any finite simplicial set $K$, and that
\w[,]{\fX\lin{P\bB}:=P\fX\lin{\bB}} we require that $\fX$ preserve the
right hand pullback squares of \wref{eqpathloop} for all \w[,]{\bX\in\bT}
so that we also have a natural identification
$$
\fX\lin{\Omega\bB}~=~\Omega\fX\lin{\bB}~.
$$
\item Any cofibration \w{i:\bB\hra\bB'} in $\bT$ induces an inclusion
\w{i\sb{\#}:\fX\lin{\bB}\hra\fX\lin{\bB'}} for all \w[.]{\bB\in\bT}
\end{enumerate}

The category of \Tma[s] will be denoted by \w[.]{\MT}
\end{defn}

\begin{defn}\label{dmapa}
For a given object \w[,]{\bY\in\C} we have a \emph{realizable} \Tma
\w{\fMT\bY} defined for any \w{\bB\in\bT} by
\w[.]{\fMT\bY\lin{\bB}:=\map\sb{\C}(\bY,\bB)} When \w{\C=\Sa} and
\w[,]{\bT=\TsR} we shall denote this by \w[.]{\fMR\bY}
The realizable \Tma \w{\fMT\bB} for \w{\bB\in\bT} will be called \emph{free}.
\end{defn}

\begin{lemma}[cf.\ \protect{\cite[8.17]{BBlaC}}]\label{lfreema}
If $\fY$ is an \Tma and \w{\fMT\bB} is a free \Tma (for \w[),]{\bB\in\bT}
there is a natural isomorphism
$$
\Phi:\map\sb{\MT}(\fMT\bB,\fY)~\xra{\cong}~\fY\lin{\bB}~,
$$
\noindent with
$$
\Phi(\fff)=\fff(\Id\sb{\bB})\in\fY\lin{\bB}\sb{0}~~\text{for any}~~
\fff\in\Hom\sb{\MT}(\fMT\bB,\fY)=\map\sb{\MT}(\fMT\bB,\fY)\sb{0}~.
$$
\end{lemma}

\begin{proof}
This follows from the strong Yoneda Lemma for enriched categories (see
\cite[2.4]{GKellyEC}).
\end{proof}

\begin{defn}\label{damastruc}
Given an $\Omega$-spectrum \w{\eA=(\uA{n})\sb{n\in\bZ}} in a model category
$\C$ as in \S \ref{snac}, we have an associated enriched sketch \w[,]{\bT=\TsA}
whose objects are of the form \w{\bB=\prod_{i\in I}\uA{n\sb{i}}} with
\w[.]{|I|<\lambda} In this case an \emph{\Ama structure} on a simplicial set
$X$ is a \Tma $\fX$ with
\begin{myeq}\label{eqamastruc}
\fX(\prod\sb{i\in I}\uA{n_{i}})=\prod\sb{i\in I}\Omega\sp{-n_{i}}X~,
\end{myeq}
\noindent which implicitly involves choices of deloopings of
\w[.]{X=\fX\lin{\uA{0}}}

When $\eA$ is only one of a set $\F$ of $\Omega$-spectra, an
\emph{\Fma structure} on $X$ is a \Fma $\fX$ satisfying
\wref[,]{eqamastruc} though in this case there is no simple formula
for describing \w{\fX\lin{\uB{n}}} in terms of \w{\fX\lin{\uA{0}}=X} for
\w[.]{\eA\neq\eB\in\F}
\end{defn}

\begin{remark}\label{rstrict}
The \Tma structures we define here are rigid, in the sense that the
action of the mapping spaces between objects of $\bT$ is strict.
In particular, the fact that \w{X=\map_{\C}(\bY,\uA{0})} has a \Ama structure
has no implications for any \w[.]{X'\simeq X} This defect can be remedied by
defining a suitable notion of a \emph{lax} \Tma[,] as was done in the
dual case in \cite[\S 6]{BBDoraR}.
\end{remark}

\begin{mysubsection}{The associated algebraic sketch}
\label{saos}
To any enriched sketch $\bT$ in a simplicial model category $\C$ we can
associate an \emph{algebraic} $\fG$-sketch \w[,]{\Theta:=\pi\sb{0}\bT}
with the same objects as $\bT$, where
\w[.]{\Hom\sb{\Theta}(A,B):=\pi\sb{0}\map\sb{\bT}(A,B)}

A \Tal \w{\Lambda:\pi\sb{0}\bT\to\Seta} is called \emph{enrichable} if it is
of the form \w{\Lambda\sb{\fX}:=\pi\sb{0}\fX} for some \w{\fX\in\MT} (not
necessarily unique).

We define a  map of \Tma[s] \w{f:\fX\to\fY} to be a \emph{weak equivalence}
if it induces an isomorphism \w{f\sb{\#}:\pi\sb{0}\fX\to\pi\sb{0}\fY} of
the corresponding \Tal[s.]
\end{mysubsection}

\begin{defn}\label{dfralg}
A \Tal \w{\Lambda:\pi\sb{0}\bT\to\Seta} is \emph{realizable}
(in $\C$) if it is enrichable by a \emph{realizable} \Tma \w{\fMT\bY} \wwh
that is, \w{\Lambda\cong\Lambda\sb{\fMT\bY}} for some \w{\bY\in\C} (again,
not necessarily unique). In this case we say that $Y$ \emph{realizes} $\Lambda$.
Any \Tal of the form \w[,]{\pi\sb{0}\fX} where $\fX$ is a free \Tma[,] will
be called a \emph{free \Tal[.]}
\end{defn}

\begin{remark}\label{rfralg}
In principle, any coproduct of free \Tma[s] or \Tal[s] is also free (in the
sense of being in the image of the left adjoint of an appropriate forgetful
functor). However, in order to avoid the question of realizability for
arbitrary free \Tma[s] (or \Tal[s),] we restrict attention to coproducts
of monogenic objects of cardinality \www[.]{<\lambda}
\end{remark}

We also have an algebraic version of Lemma \ref{lfreema}, which follows
from the usual Yoneda Lemma:

\begin{lemma}\label{lfreeta}
If $\Lambda$ is an \Tal and \w{\pi_{0}\fMT\bB} is a free \Tal
(for \w[),]{\bB\in\bT} there is a natural isomorphism
\w[.]{\Hom\sb{\TAlg}(\pi_{0}\fMT\bB,\,\Lambda)\cong\Lambda\lin{\bB}}
\end{lemma}

\begin{example}\label{egrtal}
For any ring $R$, with \w{\TsR} as in \S \ref{egeth}, we obtain an
$\fG$-sketch \w[,]{\TR:=\pi\sb{0}\TsR} namely, the full subcategory of
\w{\ho\Sa} whose objects are finite type $R$-GEMs (which are abelian group
objects in \w[).]{\ho\cS} Note that the cohomology functor \w{\HuR{-}} in
fact lands in \w[,]{\TRA} and realizable \TRal[s] are those which correspond
to actual spaces.
\end{example}

\begin{mysubsection}[\label{stma}]{Simplicial \Tma[s]}
Unfortunately, there seems to be no useful model category structure on the
category \w{\MT} of \Tma[s.] However, we do have a model category  structure
on the functor category \w[,]{\Sa\sp{\bT}} in which weak equivalences and
fibrations  are defined objectwise (cf.\ \cite[\S 7]{DKanS},
and compare \cite[\S 8]{BBlaC}), and any free \Tma is in fact a homotopy
cogroup object for \w[.]{\Sa\sp{\TsA}}
Thus we obtain a resolution model category structure (cf.\ \cite{BousC,JardB})
on the category \w{s\Sa\sp{\bT}} of \emph{simplicial} simplicially-enriched
functors \w[,]{\bT\to\Sa} in which a map \w{f:\fVd\to\fWd} is a
\emph{weak equivalence} if for each \w[,]{\bB\in\bT} the map
\w{\pi\sb{0}\fVd\lin{\bB}\to\pi\sb{0}\fWd\lin{\bB}} is a weak equivalence of
simplicial groups (see \cite[\S 2.2]{BBlaH}).
\end{mysubsection}

%
%
\sect{Mapping algebras and cosimplicial resolutions}
\label{cmacr}

We now explain how the \Ama structure on \w{X=\mapa(\bY,\bA)} described in
\S \ref{damastruc} suffices to construct a cosimplicial resolution of
$\bY$ from $X$, generalizing the results of
\cite[\S 3]{BBlaH} for \w{\bT=\TsR} when $R$ is a field.

Let $\F$ be a set of $\Omega$-spectra. By construction we have a
functor \w{\pi\sb{0}:\MF\to\FAlg} associating to any
\Fma $\fX$ its \Fal \w{\pi\sb{0}\fX} (cf. \S \ref{saos}). Since \Fma[s]
are rather complicated objects, it is natural to ask whether this functor
factors through some simpler category.  Evidently, for any fibrant simplicial
set $K$, \w{\pi\sb{0}K} depends only on the $0$-simplices and their homotopies,
i.e., on the $1$-truncation \w{\tau\sb{1}K} of $K$. However, we need even
less information if $K$ is a group object in \w[:]{\ho\Sa}

\begin{defn}\label{dhgroup}
In a model category $\C$ as in \S \ref{snac}, an
\emph{$H$-group} is a fibrant (and cofibrant) homotopy group object \wh
that is, an object \w{\bX\in\C} equipped with structure maps
\w{\mu:\bX\times\bX\to\bX} and \w[,]{(-)\sp{-1}\bX\to\bX} (with
\w{\ast\hra\bX} as the identity element), as well as chosen homotopies for
each of the identities in $\fG$ (cf.\ \S \ref{eggalg}) such as:
\w{H:\mu\circ(\mu\times\Id)\sim\mu\circ(\Id\times\mu)} for the associativity,
\w{G:\mu\circ((-)\sp{-1}\times\Id)\circ\diag\sim c\sb{\ast}} for the left
inverse, and so on.

For any cofibrant \w[,]{\bY\in\C} the Kan complex \w{K:=\map\sb{\C}(\bY,\bX)}
inherits an $H$-group structure in \w[.]{\Sa} We define an equivalence relation
on its $0$-simplices by setting:
\begin{myeq}\label{eqpizero}
f\sim g ~\Longleftrightarrow~\exists \alpha\in K\sb{1}~~\text{such that}~~
d\sb{0}\alpha=\mu\sb{\ast}(f,g\sp{-1})~~\text{and}~~d\sb{1}\alpha=\ast~.
\end{myeq}
\noindent for \w[.]{f,g,\in K\sb{0}} Note that \w{\alpha\in(PK)\sb{0}}
(see \S \ref{dpathloop}).
\end{defn}

\begin{lemma}\label{lhgroup}
For any $H$-group \w{\bX\in\C} as above,
$$
\pi\sb{0}\map\sb{\C}(\bY,\bX)~\cong~(\map\sb{\C}(\bY,\bX))\sb{0}/\sim~.
$$
\end{lemma}

\begin{defn}\label{ddiscama}
If \w{\F=\{\eA\sp{i}\}\sb{i\in I}} is some set of $\Omega$-spectra in a
pointed model category $\C$,
and \w{\hF:=\{\uA{n}\sp{i}\}\sb{\eA\sp{i}\in\F,n\in\ZZ}} as above,
a \emph{\dFma} is a function \w[,]{\fY:\hF\to\Seta\sp{J}} written
\w{\uA{n}\mapsto (P\fY\lin{\uA{n}\sp{i}}\xra{p\sb{n}}\fY\lin{\uA{n}\sp{i}})}
\wb[,]{\eA\in\F,n\in\ZZ} where $J$ denotes the single-arrow category \w[.]{0\to 1}
The category of \dFma[s] will be denoted by \w[.]{\MFrd}
\end{defn}

\begin{defn}\label{ddiscam}
Let \w{\rho:\Sk\to\Seta\sp{J}} be the functor assigning to a Kan
complex $K$ the $0$-simplices \w{p_{0}:(PK)\sb{0}\to K\sb{0}} of its path
fibration \w{p:PK\to K} (cf.\ \S \ref{dpathloop}).

If $\fX$ is an \Fma for $\F$ as above, the associated \dFma \w{\rho\fX} is
defined by setting
$$
\rho\fX\lin{\uA{n}}~:=~~p_{0}:\fX\lin{P\uA{n}}\sb{0}~\to~\fX\lin{\uA{n}}\sb{0}~.
$$
\noindent This defines a functor \w[,]{\rho:\MF\to\MFrd} since $\fX$ takes
values in Kan complexes. Moreover, we may define a covariant functor
\w{\LA:\C\to\Mop} by \w{(\LA\bY)\lin{\uA{n}}:=\rho\fMF\bY\lin{\uA{n}}}
for each \w[.]{\eA\in\F}
\end{defn}

\begin{remark}\label{ropposite}
Note that \w{\MFrd} is just the diagram category \w[,]{\SGa} indexed by
a linear category $\Gamma$ consisting of a single non-identity arrow
\w{q\sb{n}:P\uA{n}\to\uA{n}} for each \w{n\in\ZZ} and \w[,]{\eA\in\F}
and thus no non-trivial compositions.

Therefore, we have a pullback diagram:
\mydiagram[\label{eqmapsetgamma}]{
\ar @{} [dr] |<<<<<{\framebox{\scriptsize{PB}}}
\Hom\sb{\SGa}(\fX,\fY) \ar[r] \ar[d] &
\prod\sb{\eA\in\F}\prod\sb{n\in\ZZ}
\Hom\sb{\Seta}(\fX\lin{P\uA{n}},\fY\lin{P\uA{n}})
\ar[d]\sp{\fY\lin{q\sb{n}}\sb{\ast}}\\
\prod\sb{\eA\in\F} \prod\sb{n\in\ZZ}
\Hom\sb{\Seta}(\fX\lin{\uA{n}},\fY\lin{\uA{n}})
\ar[r]\sp{\fX\lin{q\sb{n}}\sp{\ast}} &
\prod\sb{\eA\in\F}\prod\sb{n\in\ZZ}
\Hom\sb{\Seta}(\fX\lin{P\uA{n}},\fY\lin{\uA{n}})
}
\noindent for any \w[.]{\fX,\fY\in\SGa}  Thus \w{\Hom\sb{\SGa}(\fX,\fY)}
is a product over \w{n\in\ZZ} and \w{\eA\in\F} (of pullback squares).
\end{remark}

We deduce from Lemma \ref{lhgroup} and the fact that each \w{\uA{n}} in $\eA$
is an (infinite) loop space:

\begin{lemma}\label{lpizerohgp}
For any \Fma $\fX$, the \Fal \w{\pi\sb{0}\fX} is determined by \w{\rho\fX}
(together with the maps
\w[).]{\mu\sb{\ast}:\fX\lin{\uA{n}\times\uA{n}}\sb{0}\to\rho\fX\lin{\uA{n}}\sb{0}}
\end{lemma}

\begin{mysubsection}{The dual Stover construction}
\label{sdsc}
In \cite{StoV}, Stover described a certain comonad on topological
spaces which eventually was shown to produce simplicial resolutions in
the \ww{E\sp{2}}-model category of \cite{DKStE}.  We now give a more conceptual
description of the dual construction.

Since the functor \w{\LA:\Sa\to\Mop} of \S \ref{ddiscam} has sufficient
information to calculate the homotopy groups
\w{\pi\sb{0}(\fMA\bY)\lin{\uA{n}}\cong[\bY,\uA{n}]} \wwh that is, the
$\eA$-cohomology of $\bY$ for each \w{\eA\in\F} \wwh if we can construct a
right adjoint \w{\RA:\Mop\to\Sa} to \w[,]{\LA} taking value in $\G$-injectives
for some class \w[,]{\G\supseteq\F} we might be able to use it to produce
a cosimplicial weak $\G$-resolution of any space $\bY$, and thus
its $\G$-completion \w[.]{\LG\bY}

Since \w{\LA} is contravariant, we need a natural isomorphism
\begin{myeq}\label{eqadjunct}
\Hom\sb{\C}(\bY,\RA\fX)~\cong~\Hom\sb{\Mop}(\LA\bY,\fX)~\cong~
\Hom\sb{\SGa}(\fX,\LA\bY)
\end{myeq}
\noindent for any \w{\bY\in\C} and \dFma $\fX$, where \w[.]{\SGa=\MFrd}

By \wref[,]{eqmapsetgamma} the right hand side naturally splits as a product
over \w{\eA\in\F} and \w{n\in\ZZ} of pullback squares in \w[:]{\Seta}
\mydiagram[\label{eqsmallpb}]{
\ar @{} [drr] |<<<<<<<<<<<{\framebox{\scriptsize{PB}}}
M\sb{n} \ar[rr] \ar[d] && \Hom\sb{\Seta}(\fX\lin{P\uA{n}},\LA\bY\lin{P\uA{n}})
\ar[d]\sp{\fX\lin{q\sb{n}}\sb{\ast}}\\
\Hom\sb{\Seta}(\fX\lin{\uA{n}},\LA\bY\lin{\uA{n}})
\ar[rr]\sp{\fX\lin{q\sb{n}}\sp{\ast}} &&
\Hom\sb{\Seta}(\fX\lin{P\uA{n}},\LA\bY\lin{\uA{n}})~.
}
\noindent Since each pointed set \w{\fX\lin{P\uA{n}}} is a coproduct in
\w{\Seta} of its non-zero element singletons, we can re-write this as
\mydiagram[\label{eqsmllpb}]{
\ar @{} [drr] |<<<<<<<<<<<{\framebox{\scriptsize{PB}}}
M\sb{n} \ar[rr] \ar[d] && \prod\sb{\fX\lin{P\uA{n}}}~\Hom\sb{\C}(\bY,P\uA{n})
\ar[d]\sp{(p\sb{\uA{n}})\sb{\#}}\\
\prod\sb{\fX\lin{\uA{n}}}~\Hom\sb{\C}(\bY,\uA{n})
\ar[rr]\sp{\fX\lin{q\sb{n}}\sp{\ast}} &&
\prod\sb{\fX\lin{P\uA{n}}}~\Hom\sb{\C}(\bY,\uA{n})~.
}
\noindent using the convention that the factor indexed by the basepoint of
\w{\fX\lin{P\uA{n}}} or \w{\fX\lin{\uA{n}}} is identified to zero.

Therefore, the right-hand side of \wref{eqadjunct} splits naturally
as a product over \w{\eA\in\F} and \w{n\in\ZZ} of certain pointed sets,
each of which factors in turn as a product of two types of pointed sets:
namely, the products sets
\begin{myeq}\label{eqprodone}
\prod\sb{\fX\lin{\uA{n}}\setminus\Image\fX\lin{q\sb{n}}}~
\Hom\sb{\C}(\bY,\uA{n})~~\times~~
\prod\sb{\ast\neq\Phi\in\fX\lin{q\sb{n}}\sp{-1}(\ast)}~
\Hom\sb{\C}(\bY,\Omega\uA{n})
\end{myeq}
\noindent and pullback squares of the form:
\mydiagram[\label{eqsmallerpb}]{
M_{\phi}\ar[rr] \ar[d] &&
\prod\sb{\fX\lin{q\sb{n}}\sp{-1}(\phi)}~\Hom\sb{\C}(\bY,P\uA{n})
\ar[d]\sp{(p\sb{\uA{n}})\sb{\#}}\\
\Hom\sb{\C}(\bY,\uA{n})\ar[rr] &&
\prod\sb{\fX\lin{q\sb{n}}\sp{-1}(\phi)}~\Hom\sb{\C}(\bY,\uA{n})~,
}
\noindent for each
\w[.]{\ast\neq\phi\in\Image\fX\lin{q\sb{n}}\subseteq\fX\lin{\uA{n}}}

Thus we have a natural identification of the left-hand term
\w{\Hom\sb{\C}(\bY,\RA\fX)} in \wref{eqadjunct} with a limit of sets of the form
\w[.]{\Hom\sb{\C}(\bY,-)} Since \w[,]{\Hom\sb{\C}(\bY,-)} commutes with
limits, we set
\begin{myeq}\label{eqadjunction}
\RA\fX~:=~
\prod\sb{\eA\in\F}~\prod\sb{n\in\ZZ}~\prod\sb{\phi\in\fX\lin{\uA{n}}}~~~Q\sb{\phi}~,
\end{myeq}
\noindent where we define \w{Q\sb{\phi}} for \w{\phi\in\fX\lin{\uA{n}}}
as follows:

\begin{enumerate}
\renewcommand{\labelenumi}{(\alph{enumi})~}
\item If \w[,]{\ast\neq\phi\in\Image\fX\lin{q\sb{n}}} then
\w{Q\sb{\phi}} is defined by the pullback square
\mydiagram[\label{eqsmallestpb}]{
Q\sb{\phi} \ar[rr]^<<<<<<<<<<{(\xi)}\ar[d]_{\omega\sb{\phi}} &&
\prod\sb{\fX\lin{q\sb{n}}\sp{-1}(\phi)}~P\uA{n}
\ar[d]\sp<<<<{\prod p\sb{\uA{n}}}\\
\uA{n}\ar[rr]^<<<<<<<<<{\diag} && \prod\sb{\fX\lin{q\sb{n}}\sp{-1}(\phi)}~\uA{n}
}
\noindent in $\C$.
\item If \w[,]{\phi\not\in\Image\fX\lin{q\sb{n}}} we set \w[.]{Q\sb{\phi}:=\uA{n}}
\item If \w[,]{\phi=\ast} we set
$$
Q\sb{\phi}~:=~\prod\sb{\fX\lin{q\sb{n}}\sp{-1}(\ast)\setminus\{\ast\}}~
\Omega\uA{n}~.
$$
\end{enumerate}
\end{mysubsection}

\begin{defn}\label{dmonad}
We call \w{\RA\fY} of \wref{eqadjunction} the \emph{dual Stover construction}
for $\F$, applied to the \dFma $\fY$. This defines a functor
\w[,]{\RA:\Mop\to\C} right adjoint to \w[,]{\LA} and thus a monad
\w{\TF:=\RA\circ\LA} on $\C$, with unit
\w{\eta=\widehat{\Id\sb{\LA}}:\Id\to\TF}
and multiplication \w[,]{\mu=\RA\widetilde{\Id\sb{\TF}}:\TF\circ\TF\to\TF}
as well as a comonad \w{\SA:=\LA\circ\RA} on \w[,]{\Mop} with counit
\w{\vare=\widetilde{\Id\sb{\RA}}:\SA\to\Id} and comultiplication
\w{\delta=\LA\widehat{\Id\sb{\SA}}:\SA\to\SA\circ\SA}
(cf.\ \cite[VI, \S 1]{MacLC}).

Recall that a \emph{coalgebra} over \w{\SA} is an object \w{\fY\in\Mop}
equipped with a section \w{\zeta:\fY\to\SA\fY} for the counit
\w[,]{\vare\sb{\fY}:\SA\fY\to\fY} such that
\begin{myeq}\label{eqcoalg}
\SA\zeta\circ\zeta~=~\delta\sb{\fY}\circ\zeta
\end{myeq}
\noindent (see \cite[VI, \S 2]{MacLC}).
\end{defn}

\begin{example}\label{egcoalg}
Any \emph{realizable} \dFma \w{\LA\bY=\rho\fMF\bY} has a canonical structure of
a coalgebra over \w[,]{\SA} with \w{\zeta:\LA\bY\to\LA\RA\LA\bY}
(in \w[)]{\Mop} equal to \w[,]{\LA(\widehat{\Id\sb{\LA\bY}})} where
\w{\widehat{\Id\sb{\LA\bY}}:\bY\to\RA\LA\bY} is the adjoint of
\w[.]{\Id:\LA\bY\to\LA\bY}
\end{example}

\begin{mysubsection}{The Stover category}
\label{ssosc}
If $\F$ is a set of $\Omega$-spectra in $\C$, and $\kappa$
is some (infinite) cardinal, we define an \emph{$\F$-Stover object}
(for $\kappa$) to be a product of at most $\kappa$ objects which are either
spaces \w{\uA{n}} from an $\Omega$-spectrum \w[,]{\eA\in\F} or are pullbacks
$Q$ of the form:
\mydiagram[\label{eqstoverpb}]{
Q \ar[rr]\ar[d] && \prod\sb{T}~P\uA{n}\ar[d]\sp{p}\\
\uA{n}\ar[rr]\sp{\diag} && \prod\sb{T}~\uA{n}
}
\noindent for various \w[,]{\eA\in\F} where $T$ is some set of
cardinality $\leq\kappa$.

Thus if $\kappa$ bounds the cardinality of all sets \w{\fX\lin{\uA{n}}}
\wb[,]{\eA\in\F, n\in\ZZ} for some \dFma $\fX$, then \w{\RA\fX} is an
$\F$-Stover object for $\kappa$. Moreover, since each pullback $Q$ in
\wref{eqstoverpb} is weakly equivalent to a product of copies of
\w[,]{\Omega\uA{n}} we see that any $\F$-Stover object is $\G$-injective for
any class \w[.]{\G\supseteq\F}

We denote by \w{\TFs{\kappa}} the full simplicial subcategory of $\C$
consisting of all $\F$-Stover objects for $\kappa$. We shall always assume
that \w[,]{\TsF\subseteq\TFs{\kappa}} which holds as long as
\w{\kappa\geq\lambda} (see  Remark \ref{rlambdaone}).
Note that \w{\TFs{\kappa}}  is itself an enriched sketch, whose \ma[s] will be called
\emph{\Sma[s]} (for $\kappa$). The category of \Sma[s] will be denoted by
\w[,]{\MAS} and the corresponding free \Sma functor will be written
\w[.]{\fMS:\C\to\MAS}

We thus have a forgetful functor
\w[,]{U\sp{\Stov}:\MFS\to\MF} and when there is no danger of confusion we
shall denote the composite \w{\rho\circ U\sp{\Stov}} simply by $\rho$,
so we have \w[.]{\rho\fMF=\rho\fMS}
\end{mysubsection}

\begin{remark}\label{ropcat}
It is difficult to keep track of the (co)monads obtained from adjoint
functors when they are not covariant, which is why Definition \ref{dmonad}
was given in terms of \w[.]{\Mop} However, for the purposes of the following
Lemma and Proposition, we prefer to work in \w{\MFrd} itself, so that we
will refer to \w{\vare\sb{\fY}\op:\fY\to\SA\fY} rather than the original
\w[,]{\vare\sb{\fY}:\SA\fY\to\fY} and so on.
\end{remark}

\begin{lemma}\label{lcoalg}
The coalgebra structure map \w{\zeta:\LA\bY\to\LA\RA\LA\bY}
(in \w[),]{\Mop} of a realizable \dFma \w{\LA\bY} is induced
by a map of \Sma[s] \w{\zeta':\fMS(\RA(\rho\fMS\bY))\to\fMS\bY}
(in \w[),]{\MFS} so \w[.]{\zeta\op=\rho\zeta'}
\end{lemma}

\begin{proof}
Since \w{\RA(\rho\fMS\bY)} is an $\F$-Stover object, we see that
\w{\fMS(\RA(\rho\fMS\bY))} is a free \Sma[,] so in order to define the map
\w[,]{\zeta'} by Lemma \ref{lfreema} it suffices to produce a
``tautological element'' \w{\lra{\zeta'}} in
\begin{equation*}
\begin{split}
\fMS\bY\lin{\RA(\rho\fMS\bY)}\sb{0}~=&~
\map\sb{\C}(\bY,\,\RA(\rho\fMS\bY))\sb{0}~=~
\Hom\sb{\C}(\bY,\,\RA(\rho\fMS\bY))\\
~\cong&~\Hom\sb{\Mop}(\LA\bY,\,\LA\bY)~,
\end{split}
\end{equation*}
\noindent where the last isomorphism is just \wref[.]{eqadjunct}
We may thus choose \w{\lra{\zeta'}} to be the adjoint of
\w[.]{\Id\sb{\LA\bY}}

More explicitly, since \w{\RA(\rho\fMS\bY)} is defined by a limit,
\w{\lra{\zeta'}} will be determined by choosing compatible elements in
each component of \wref[.]{eqadjunction} However, each component \w{\uA{n}}
is indexed by an element $\phi$ in \w[,]{\fMS\bY\lin{\uA{n}}\sb{0}}
and each component \w{P\uA{n}} is indexed by an element $\Phi$ in
\w[,]{\fMS\bY\lin{P\uA{n}}\sb{0}} and all these choice are compatible,
yielding the required \w[.]{\lra{\zeta'}}
\end{proof}

\begin{prop}\label{pcoalgebra}
For every \Sma $\fX$, the corresponding \dFma \w{\fY:=\rho\fX} has a
natural structure of a coalgebra over \w[.]{\SA}
\end{prop}

Compare \cite[Proposition 3.7]{BBDoraR} and \cite[Proposition 3.16]{BBlaH}.

\begin{proof}
Using Remark \ref{ropcat}, we must produce a section \w{\zeta\op:\SA\fY\to\fY}
for \w{\vare\sb{\fY}\op:\fY\to\SA\fY} satisfying \wref[.]{eqcoalg}
In fact, we will construct a map of \Sma[s]
\w{\zeta':\fMS\RA\fX\to\fX} fitting into a commuting square
\mydiagram[\label{eqcoalgebra}]{
\fMS(\RA(\rho\fMS\RA\fY)) \ar[d]\sb{\delta'} \ar[rrr]^<<<<<<<<<{\fMS\zeta'}
&&& \fMS\RA\fY \ar[d]\sb{\zeta'}\\
\fMS\RA\fY \ar[rrr]\sp{\zeta'} &&& \rho\fY~.
}
\noindent in \w[,]{\MA} with \w{\zeta=(\rho\zeta')\op} and
\w[.]{\delta\sb{\fY}=\rho\delta'}

As in the proof of Lemma \ref{lcoalg}, in order to define the map
$$
\delta':\fMS(\RA(\rho\fMS\RA\fY))~\to~\fMS\RA\fY~,
$$
by Lemma \ref{lfreema} it suffices to produce a ``tautological element''
\w{\lra{\delta'}} in
$$
\fMS\RA\fY\lin{\RA(\rho\fMS(\RA\fY))}\sb{0} ~=~
\map\sb{\C}(\RA\fY,\,\RA(\rho\fMS(\RA\fY)))\sb{0}~,
$$
\noindent that is, a map \w[,]{\theta:\RA\fY\to\RA((\rho\fMS(\RA\fY))}
which we choose to be the adjoint of
\w[.]{\Id:\rho\fMS(\RA\fY)\to\rho\fMS(\RA\fY)}

The verification that \wref{eqcoalgebra} commutes is as in
\cite[Proposition 3.16]{BBlaH}.
\end{proof}

\begin{notation}\label{nkappa}
Given a set $\F$ of $\Omega$-spectra and an \Fma $\fX$, we write
$$
\kappa\sb{\fX}~:=~\sup~\{|\fX\lin{\uA{n}}|~:~\eA\in\F, n\in\ZZ\}~.
$$
\end{notation}

\begin{prop}\label{presma}
Assume given a set $\F$ of $\Omega$-spectra in a model
category $\C$ (as in \S \ref{snac}), and an \Fma $\fX$ which extends to an
\Sma for some cardinal \w[.]{\kappa\geq \kappa\sb{\fX}} Then there is a
cosimplicial object \w[,]{\Wu\in c\C} with each \w{\bW\sp{n}} an $\F$-Stover
object for $\kappa$, equipped with a weak equivalence of simplicial
\Fma[s] \w[.]{\fMF\Wu\to\cd{\fX}}
\end{prop}

\begin{proof}
The fact that the \emph{discrete} $\F$-mapping algebra \w{\rho\fMF\Wu} is
homotopy equivalent to \w{\cd{\rho\fX}} follows from the usual properties of the
``standard construction'' for a (co)monad (cf.\ \cite[App.,\S 3]{GodeT} and
compare \cite[\S 8.6]{WeibHA}).
However, since we need the precise details of the contravariant case in our setting,
we first summarize these as follows:

We may define an augmented simplicial object \w{\fVdd\to\fY}
in \w{\Mop} by iterating the corresponding comonad \w{\SA:\Mop\to\Mop} on
\w[,]{\fY:=\rho\fX} so \w{\fVt{k}=\SA\sp{k+1}\fY} (cf.\ \cite[\S 8.6.4]{WeibHA})
By Proposition \ref{pcoalgebra}, $\fY$ is a coalgebra over \w[,]{\SA}
so it is \ww{\SA}-projective (cf.\ \cite[\S 8.6.6]{WeibHA}), and the structure
map \w{\zeta\sb{\fY}:\fY\to\tilde{\fV}\sb{0}=\SA\fY} provides an extra
degeneracy map \w{s\sb{k+1}:=\SA\sp{k+1}\zeta\sb{\fY}:\fVt{k}\to\fVt{k+1}}
(see \cite[Proposition 8.6.8]{WeibHA}).

An explicit description of \w{\fVdd\to\fY} in low dimensions is given by
the following diagram in \w[:]{\Mop}
\mydiagram[\label{eqsimpma}]{
\fVt{-1}:=\fY \ar@/^{1.0pc}/[rrr]\sp{s\sb{0}=\zeta\sb{\fY}} &&&
\fVt{0}:=\SA\fY \ar@/^{1.0pc}/[lll]\sp{d\sb{0}=\vare\sb{\fY}=
\widetilde{\Id\sb{\RA\fY}}}
\ar@/^{3.0pc}/[rrrr]\sp{s\sb{0}=\delta\sb{\fY}=\LA\widehat{\Id\sb{\SA\fY}}}
\ar@/^{1.0pc}/[rrrr]\sp{s\sb{1}=\SA\zeta\sb{\fY}}  &&&&
\fVt{1}:=\SA\sp{2}\fY \ar@/^{1.0pc}/[llll]\sb{d\sb{0}=\vare\sb{\SA\fY}}
\ar@/^{3.0pc}/[llll]\sb{d\sb{1}=\SA\vare\sb{\fY}}
}

Applying the functor \w{\RA} dimensionwise to \w{\fVdd\to\fY} yields an
augmented simplicial object \w[.]{\RA\fVdd\to\RA\fY} We set
$$
\bW\sp{n}~:=~\RA\tilde{\fV}\sb{n-1}~=~\RA\SA\sp{n}\fY\hsp \text{with}\hsm
\bW\sp{0}~:=~\RA\fY~.
$$
\noindent Moreover, we have an extra map
\w[,]{d\sp{n+1}:\bW\sp{n}=\RA\SA\sp{n}\fY\to\RA\SA\sp{n+1}\fY=\bW\sp{n+1}}
adjoint to \w[,]{\Id\sb{\SA\sp{n+1}\fY}} so we actually obtain a cosimplicial
object \w[.]{\Wu}

Once again, we have an explicit description of \w{\Wu} in low dimensions
as a diagram in $\C$:
\mydiagram[\label{eqcosimpsp}]{
\bW\sp{0}=\RA\fY \ar@/^{2.5pc}/[rr]\sp{d\sp{0}=\RA\zeta\sb{\fY}}
\ar[rr]\sp<<<<<<<<<{d\sp{1}=\widehat{\Id\sb{\SA\fY}}} &&
\bW\sp{1}=\RA\SA\fY \ar@/^{2.0pc}/[ll]\sb{s\sp{0}=\RA\vare\sb{\fY}}
\ar@/^{4.0pc}/[rrr]\sp{d\sp{0}=\RA\LA\widehat{\Id\sb{\SA\sp{2}\fY}}}
\ar@/^{2.0pc}/[rrr]\sp{d\sp{1}=\RA\SA\zeta\sb{\fY}}
\ar[rrr]\sp{d\sp{2}=\widehat{\Id\sb{\SA\SA\fY}}} &&&
\bW\sp{2}=\RA\SA\sp{2}\fY
\ar@/^{2.0pc}/[lll]\sb{s\sp{0}=\RA\vare\sb{\SA\fY}}
\ar@/^{4.0pc}/[lll]\sb{s\sp{1}=\RA\SA\vare\sb{\fY}}
}

As we saw in the proof of Proposition \ref{pcoalgebra}, the opposite of
\ww{\SA}-coalgebra structure map of \w{\fY=\rho\fX} lifts to a map of \Sma[s]
\w[,]{\zeta':\fMS\bW\sp{0}\to\fX} making \w{\fMA\Wu\to\fX} into an augmented
simplicial \Sma[,] which is described in low dimensions by the diagram
(in \w[):]{\MAS}

\mydiagram[\label{eqsimpsma}]{
\fX &&& \fMS\bW\sp{0} \ar[lll]\sb{\zeta'}
\ar@/^{2.0pc}/[rrr]\sp{s\sb{0}=\fMS\RA\vare\sb{\fY}} &&&
\fMS\bW\sp{1} \ar[lll]\sb<<<<<<<<<{d\sb{0}=\fMS\RA\zeta\sb{\fY}}
\ar@/^{2.0pc}/[lll]\sb{d\sb{1}=\fMS\widehat{\Id\sb{\SA\fY}}} & \dotsc
}
\noindent By construction, applying $\rho$ to this yields
\w{\fVdd\to\fY} (viewed as a coaugmented cosimplicial object \w[).]{\Mop}
Since each \w{\uA{n}} of $\eA$ is an (infinite) loop space, for each
\w[,]{n,k\geq 0} each of the pointed sets \w{\fVt{k}\lin{\uA{n}}} (in the
notation of \S \ref{ropposite}) is equipped with a binary operation
\w{\mu\sb{\#}:\fVt{k}\lin{\uA{n}}\times
\fVt{k}\lin{\uA{n}}\to\fVt{k}\lin{\uA{n}}}
(induced by the $H$-space multiplication \w[),]{\mu:\uA{n}\times\uA{n}\to\uA{n}}
and we can therefore use Lemma \ref{lhgroup} to calculate the homotopy groups
of the augmented simplicial (abelian) group
$$
\Gd~:=~\pi\sb{0}(\fMS\Wu\lin{\uA{n}})~\xra{\zeta'\lin{\uA{n}}\sb{\#}}~
\pi\sb{0}\fX\lin{\uA{n}}~=:~G\sb{-1}
$$
\noindent from the object \w{\fVdd\lin{P\uA{n}\to\uA{n}}} in \w[.]{s\Seta\sp{J}}

Note that the indexing of \wref{eqsimpsma} does not match with that of
\wref[,]{eqsimpma} even after taking into account the fact that these diagrams
are in opposite categories.  It will be convenient to choose our indexing
convention for the augmented simplicial object \w{\rho\fMS\Wu\to\fY}
(in \w[)]{s\MAS} so that the extra degeneracy map is indexed as:
$$
s\sb{-1}\sp{k}~=~(\vare\sb{\SA\sp{k+1}\fY})\op:\rho\fMS\bW\sp{k}~\to~
\rho\fMS\bW\sp{k+1}
$$
\noindent with
\begin{myeq}\label{eqsminus}
d\sb{i}s\sb{-1}~=~\begin{cases} \Id & \hsm\text{if}~i=0\\
s\sb{-1}d\sb{i-1} & \hsm\text{if}~i\geq 1
\end{cases}
\end{myeq}
\noindent (which follows from the fact that
\w{\vare\sb{\fY}\circ\zeta\sb{\fY}=\Id} in \w[,]{\Mop} and $\vare$ is a
natural transformation).
This completes the survey of the known facts about the ``standard construction''.

If we knew that \w{s\sp{k}\sb{-1}} induced a group homomorphism
\w{G\sb{k}\to G\sb{k+1}} for each \w[,]{k\geq -1} we could deduce directly
that the augmented simplicial group \w{\Gd\to G\sb{-1}} is acyclic.  However,
unlike the rest of the simplicial structure on \w[,]{\fVdd} the maps \w{s\sb{-1}}
are not induced from maps of full \Sma[s,] so they need not respect the $H$-space
structure induced by that of \w[.]{\uA{n}}

Nevertheless, we can modify the usual proof of the acyclicity of
\w{\Gd\to G\sb{-1}} as follows:  any Moore cycle
\w{\alpha\in C\sb{k}\Gd\subseteq G\sb{k}} (with \w{d\sb{i}\alpha=0} for
\w[)]{0\leq i\leq k} may be represented by an element
\w[,]{a\in\fVt{k}\lin{\uA{n}}} equipped with elements
\w{b\sb{i}\in\fVt{k-1}\lin{P\uA{n}}} such that
\w{d\sb{i}a=(\fVt{k-1}\lin{q\sb{n}})(b\sb{i})} \wb[.]{0\leq i\leq k}

If we let \w[,]{c:=s\sb{-1}\sp{k}a\in\fVt{k+1}\lin{\uA{n}}} by \wref{eqsminus}
we have \w{d\sb{0}(c)=a} and
$$
d\sb{i}c~=~d\sb{i}s\sb{-1}\sp{k}a~=s\sb{-1}\sp{k-1}d\sb{i}a~=~
s\sb{-1}\sp{k-1}(\fVt{k-1}\lin{q\sb{n}})(b\sb{i})~=~
(\fVt{k-1}\lin{q\sb{n}})(s\sb{-1}\sp{k-1}b\sb{i})\hsm\text{for}\hsm 1\leq i~,
$$
\noindent so that \w{[c]\in\pi\sb{0}\fVt{k+1}\lin{\uA{n}}=G\sb{k+1}} is
a Moore chain, with \w[.]{\partial\sb{k+1}([c]]=\alpha}

Thus we have shown that \w{\zeta':\fMS\Wu\to\cd{\fX}} is a weak equivalence
of simplicial \Sma[s.]
\end{proof}

\begin{corollary}\label{cresma}
Given a set $\F$ of $\Omega$-spectra in a model category $\C$,
and an object \w[,]{\bY\in\C} let \w{\Wu\in c\C} be obtained from the \Sma
\w{\fX=\fMS\bY}  as in Proposition \ref{presma}.
Then for every \w{\eA=(\uA{k})\sb{k\in\ZZ}} in $\F$ and \w[,]{n\in\ZZ}
the augmented simplicial abelian group \w{[\Wu,\uA{n}]\to[\bY,\uA{n}]}
is acyclic.
\end{corollary}

Although the construction above works for any family of $\Omega$-spectra $\F$,
its relevance to our original question on recovering $\bY$ from
\w{\map\sb{\C}(\bY,\bA)} is restricted to the following special case:

\begin{thm}\label{tresm}
Assume that $R$ is a field, \w{\bY\in\Sa} is $R$-good (cf.\ \cite[I, 5.1]{BKanH}),
and \w{X:=\mapa(\bY,\bA)} for \w[.]{\bA=\KR{n}} Let \w{\fX} be an
\Sma[-structure] on $X$ for \w{\F:=\{\HR{R}\}} and \w{\kappa=\kappa\sb{\fMS\bY}}
(cf.\ \S \ref{nkappa}), and let \w{\Wu} be constructed from $\fX$ as in
Proposition \ref{presma}. Then \w[,]{X\simeq\mapa(\uTot\Wu,\bA)} which means that
we can recover $\bY$ from $\fX$ up to $\bA$-equivalence.
\end{thm}

\begin{proof}
Any simplicial $R$-module is of the form
\w{M\simeq\prod\sb{n=0}\sp{\infty}\,\KP{V\sb{n}}{n}} for some $\NN$-graded
$R$-vector space \w[,]{(V\sb{n})\sb{n=0}\sp{\infty}} up to weak equivalence.
Moreover, each \w{V\sb{n}=\bigoplus\sb{\lambda\sb{n}}R} embeds as a split summand in
\w[.]{V'\sb{n}=\prod\sb{\lambda\sb{n}}R} Thus $M$ is a homotopy retract of
\w{M'=\prod\sb{n=0}\sp{\infty}\,\KP{V'\sb{n}}{n}} so that for any (augmented)
cosimplicial space \w{\bY\to\Wu} the (augmented) simplicial abelian group
\w{[\Wu,M]\to[\bY,M]} is a retract of \w[.]{[\Wu,M']\to[\bY,M']}
However, for any space $\bX$ we have
\w[,]{[\bX, M']\cong\prod\sb{n=0}\sp{\infty}\prod\sb{\lambda\sb{n}}\,[\bX,\,\KR{n}]}
so if \w{[\Wu,\KR{n}]\to[\bY,\KR{n}]} is acyclic for each \w[,]{n\geq 0}
so is \w[.]{[\Wu,M]\to[\bY,M]}

Now let $\G$ denote the class of \ww{\Omega\sp{\infty}}-spaces of
\ww{\HR{R}}-module spectra \wh or equivalently, all simplicial $R$-modules
(cf.\  \cite[IV, \S 2.7]{EKMMayR}). If \w[,]{\fX:=\fMF\bY} then the
cosimplicial resolution \w{\Wu} is a weak $\G$-resolution of $\bY$ (cf.\
\cite[Definition 6.1]{BousC}), by the above and Corollary \ref{cresma}, so
\w{\uTot\Wu\simeq R\sb{\infty}\bY} by \cite[\S 7.7 \& Theorem 9.5]{BousC}.
Therefore, \w{\bY\to\uTot\Wu} induces a weak equivalence
\w{\mapa(\bY,\KR{n})\to\mapa(\uTot\Wu,\KR{n})} by \cite[Proposition 8.5]{BousC}.
\end{proof}

%
%
\sect{Mapping algebras and completions}
\label{cmac}

By Theorem \ref{tresm}, the construction \w{\Wu} of Section \ref{cmacr}
can be used to recover $\bY$ from a realizable \Sma \w{\fMS\bY} when working
over a field. However, for more general rings or ring spectra, we need
a more elaborate choice of $\G$ consisting of \emph{all}
\ww{\HR{R}}-module spectra (or equivalently, all simplicial $R$-modules) \wh
a proper class. For this purpose, one could
restrict attention the collection of simplicial) $R$-modules of cardinality
$<\nu$, for some inaccessible cardinal $\nu$, but in order to
avoid unnecessary set-theoretic assumptions we describe a modified version of
the dual Stover construction, adapted to the particular \w{\bY\in\Sa} at hand.

\begin{defn}\label{dsimrmod}
For any commutative ring $R$, consider the category \w{s\Mod{R}} of all
simplicial $R$-modules. Let \w{\sMR} denote (a skeleton of) the subcategory
of \w{s\Mod{R}} with the same objects, but only proper monomorphisms as maps.
For every cardinal $\lambda$ we denote by \w{(\sMR)\sb{\lambda}} the full
subcategory of \w{\sMR} consisting of objects of cardinality $\leq\lambda$.
\end{defn}

\begin{mysubsect}{The modified dual Stover construction}
\label{smdsc}

When $R$ is not a field, the model category structure of Bousfield
(cf.\ \S \ref{sgrmc}) defines weak equivalences of cosimplicial spaces using the
class $\G$ of all \ww{\Omega\sp{\infty}}-spaces of \ww{\HR{R}}-module spectra.
However, since $\G$ is a proper \emph{class}, the dual Stover construction of
\S \ref{sdsc} makes no sense for \w{\hF:=\G} (cf.\ \wref[).]{eqhatf}
Thus, we need to replace $\G$ by a \emph{set} of \ww{\Omega\sp{\infty}}-spaces of
\ww{\HR{R}}-module spectra, which still suffice to recover $\bY$
from \w[.]{\fMS\bY} For this purpose we shall use the set
\w[,]{\hF:=(\sMR)\sb{\lambda}} for a suitable cardinal
\w{\lambda=\widehat{\lambda}\sb{\bY}} depending on $\bY$ \wh see
\S \ref{dbiglambda} below.
\end{mysubsect}

\begin{defn}\label{dsdiscrma}
If $R$ is a commutative ring, $\lambda$ is a cardinal, and
\w[,]{\hF:=(\sMR)\sb{\lambda}} a \emph{\dRma}
(for $\lambda$) is a functor \w[,]{\fY:\hF\to\Seta\sp{J}} written
\w{M\mapsto (P\fY\lin{M}\xra{p\sb{M}}\fY\lin{M})}
where $J$ denotes the single-arrow category \w{0\to 1} (cf.\ \S \ref{ddiscama}).
The category of \dRma[s] will be denoted by \w[.]{\MRrd}
\end{defn}

\begin{defn}\label{dsdiscrm}
Note that \w{\hF=(\sMR)\sb{\lambda}} is indeed the set of
\ww{\Omega\sp{\infty}}-spaces in a certain collection $\F$ of $\Omega$-spectra
which are \ww{\HR{R}}-module spectra, as in \wref[.]{eqhatf}

If $\fX$ is an \Fma for this $\F$, the \emph{associated} \dRma \w{\rho\fX} is
defined by setting
$$
\rho\fX\lin{M}~:=~
\left(\fX\lin{p\sb{M}}\sb{0}:\fX\lin{PM}\sb{0}~\to~\fX\lin{M}\sb{0}\right)
$$
\noindent (see \S \ref{ddiscam}).

This defines a functor \w[,]{\rho:\MF\to\MRrd} since $\fX$ takes
values in Kan complexes. Moreover, we may define a covariant functor
\w{\LR{\lambda}:\Sa\to\MRop} by \w{(\LR{\lambda}\bY)\lin{M}:=\rho\fMF\bY\lin{M}}
for each \w[.]{M\in\hF}
\end{defn}

\begin{remark}\label{rsopposite}
Note that \w{\MRrd} is just the diagram category
\w[,]{\SGa} indexed by a category $\Gamma$ whose maps consist of the arrows
in squares of the form:
\mydiagram[\label{eqsquaregamma}]{
PM\sb{0} \ar[d]^{q\sb{M\sb{0}}}\ar[rr]^{Pj} && PM\sb{1} \ar[d]^{q\sb{M\sb{1}}}\\
M\sb{0} \ar[rr]^{j} && M\sb{1}
}
\noindent for any map \w{j:M\sb{0}\to M\sb{1}} in \w[,]{(\sMR)\sb{\lambda}}
where the horizontal arrows may be composed for
\w[.]{M\sb{0}\hra M\sb{1}\hra M\sb{2}}

Therefore, we can give an explicit description for the right adjoint
\w{\RR{\lambda}:\MRop\to\Sa} to \w{\LR{\lambda}:\Sa\to\MRop} as in
\S \ref{ropposite}, with \w{\RR{\lambda}\fX} described by the limit of a diagram
described as follows:
\begin{enumerate}
\renewcommand{\labelenumi}{(\alph{enumi})~}
\item For each \w[,]{M\sb{0}\in\hF} consider the map of sets
\w[.]{\fX\lin{q\sb{M\sb{0}}}:\fX\lin{PM\sb{0}}\to\fX\lin{M\sb{0}}}
If \w{\phi\in\fX\lin{M\sb{0}}} is contained in
\w[,]{\Image(\fX\lin{q\sb{M\sb{0}}})\subseteq \fX\lin{M\sb{0}}} then for any
\w{j:M\sb{0}\hra M\sb{1}} also \w{\fX\lin{j}(\phi)\in\fX\lin{M\sb{0}}} is contained
in \w[,]{\Image(\fX\lin{q\sb{M\sb{1}}})\subseteq \fX\lin{M\sb{1}}} and we have a
``corner of corners'' of the form:
\mysdiag[\label{eqcorncorn}]{
& \prod\sb{\Phi\in(\fX\lin{q\sb{M\sb{0}}})^{-1}(\phi)}~PM\sb{0}
\ar[d]\sp<<<<{\prod p\sb{M\sb{0}}} \ar[rd]\sp<<<<<<<<{\prod Pj} && \\
M\sb{0} \ar[r]\sp<<<<<{\diag} \ar[rd]\sb{j}&
\prod\sb{\Phi\in(\fX\lin{q\sb{M\sb{0}}})^{-1}(\phi)}~M\sb{0}\ar[rd]\sb{\prod{j}} &
\prod\sb{\Phi\in(\fX\lin{q\sb{M\sb{0}}})^{-1}(\phi)}~PM\sb{1}
\ar[d]\sp<<<<{\prod p\sb{M\sb{1}}} &\\
& M\sb{1} \ar[r]\sp<<<<<<<<{\diag} &
\prod\sb{\Phi\in(\fX\lin{q\sb{M\sb{0}}})^{-1}(\phi)}~M\sb{1} &
\prod\sb{\Phi\in(\fX\lin{q\sb{M\sb{1}}})^{-1}(\phi)}~PM\sb{1}
\ar[d]\sp<<<<{\prod p\sb{M\sb{1}}} \ar[lu]\sb<<<<<<<{\prod(\fX\lin{j})\sp{\ast}} \\
&& M\sb{1} \ar[lu]\sp{=} \ar[r]\sb<<<<<<<<<<{\diag} &
\prod\sb{\Phi\in(\fX\lin{q\sb{M\sb{1}}})^{-1}(\phi)}~M\sb{1}
\ar[lu]\sp{\prod(\fX\lin{j})\sp{\ast}}
}
\noindent for each square \wref[.]{eqsquaregamma}
\item On the other hand, if \w{\phi\in\fX\lin{M\sb{0}}} is not contained in
\w[,]{\Image(\fX\lin{q\sb{M\sb{0}}})} then we simply have an arrow from the copy
of \w{M\sb{0}} indexed by $\phi$ to the copy of \w{M\sb{1}} indexed by
\w{\fX\lin{j}(\phi)} in our diagram.
\end{enumerate}
\end{remark}

\begin{defn}\label{dhmonad}
The right adjoint \w{\RR{\lambda}:\MRop\to\Sa} to \w{\LR{\lambda}} defines a monad
\w{\hTR{\lambda}:=\RR{\lambda}\circ\LR{\lambda}} on \w[,]{\Sa} with unit
\w{\het{}=\widehat{\Id\sb{\LR{\lambda}}}:\Id\to\hTR{\lambda}} and multiplication
\w[,]{\hmu=\RR{\lambda}\widetilde{\Id\sb{\hTR{\lambda}}}:
\hTR{\lambda}\circ\hTR{\lambda}\to\hTR{\lambda}}
as well as a comonad \w{\SR{\lambda}:=\LR{\lambda}\circ\RR{\lambda}} on
\w[,]{\MRop} with counit
\w{\hvar=\widetilde{\Id\sb{\RR{\lambda}}}:\SR{\lambda}\to\Id} and comultiplication
\w[.]{\hdel=\LR{\lambda}\widehat{\Id\sb{\SR{\lambda}}}:
\SR{\lambda}\to\SR{\lambda}\circ\SR{\lambda}}
\end{defn}

\begin{remark}\label{rlimits}
Since all the limits in the construction of \w{\RR{\lambda}\fX} take place in
\w[,]{\sMR} the analogue of Proposition \ref{pcoalgebra}, stating that
\w{\rho\fX} has natural structure of a coalgebra over \w[,]{\SR{\lambda}}
will hold for any \Fma $\fX$ which can be extended to a \ww{\bT\sb{\F'}}-\ma for
\w{\hF'=(\sMR)\sb{\mu}} where $\mu$ bounds the cardinalities of all objects
in \w[.]{\RR{\lambda}\fX}
\end{remark}

As in Section \ref{cmacr} we then have:

\begin{prop}\label{pcosimpres}
Assume given a commutative ring $R$ and a \Fma $\fX$ for
\w[,]{\hF:=(\sMR)\sb{\lambda}} which is a coalgebra over \w[,]{\SR{\lambda}}
and let \w{\hWlu} be obtained from $\fX$ as in Proposition \ref{presma} by iterating
\w{\hTR{\lambda}} on \w[,]{\SR{\lambda}\rho\fX} with an augmentation
\w[.]{\fMF\hWlu\to\fX} Then each \w{\hWl{n}} is in \w[,]{\sMR} and
for every \w[,]{M\in\hF} the augmented simplicial $R$-module
\w{[\hWlu,M]\to\pi\sb{0}\fX\lin{M}} is acyclic.
\end{prop}

\begin{proof}
We obtain a cosimplicial object \w{\hWlu} over \w{\Sa} using the fact that
\w{\rho\fX} is a coalgebra over the comonad \w[.]{\SR{\lambda}}
By the description in \S \ref{rsopposite} we see that for any
\dRma $\fY$, \w{\RR{\lambda}\fY} is a limit of simplicial $R$-modules (and
all the maps in  question are maps in \w[),]{s\Mod{R}} so in particular,
\w{(\hTR{\lambda})\sp{n+1}\bY:=\RR{\lambda}(\LR{\lambda}(\hTR{\lambda})\sp{n}\bY)}
is a simplicial $R$-module.

Note that for any \w[,]{M\in\hF} the functor \w{[-,M]:\Sa\to\Mod{R}}
factors through \w[,]{\LR{\lambda}(-)\lin{M}} by Lemma \ref{lhgroup},
so to show that \w{[\hWlu,M]\to\pi\sb{0}\fX\lin{M}} is acyclic, it suffices to
show that the augmented simplicial \dRma \w{\rho\fX\leftarrow\LR{\lambda}\hWlu} is
contractible, which follows from \cite[Proposition 8.6.10]{WeibHA} as in the proof
of Proposition \ref{presma}.
\end{proof}

\begin{corollary}\label{ccosimpres}
Given a commutative ring $R$, a cardinal $\lambda$,  and \w[,]{\bY\in\Sa} we
obtain a coaugmented cosimplicial space \w{\bY\to\hWlu} (by taking
\w[),]{\fX:=\fMF\bY} such that for every \w{M\in\hF} the augmented
simplicial $R$-module \w{[\hWlu,M]\to[\bY,M]} is acyclic.
\end{corollary}

Note that if $\fX$ is realizable, it automatically has an
\ww{\SR{\lambda}}-coalgebra structure, as in \S \ref{egcoalg}.

\begin{defn}\label{dsrmod}
For any simplicial $R$-module $M$, \w[,]{\bX\in\Sa} and \w{\phi:\bX\to\UM} in
\w[,]{\Sa} we denote by \w{\hIm(\phi)} the smallest simplicial submodule of $M$
containing \w[.]{\Image(\phi)} If \w[,]{\hIm(\phi)=M}
we call $\phi$ an \emph{effective epimorphism}, denoted by
\w[.]{\phi:\bX\efpic M}

Moreover, the path space construction \w{\UPM}
of \S \ref{dpathloop} is also a simplicial $R$-module, and
\w{\Phi:\bX\to \UPM} is a nullhomotopy of \w{\phi:\bX\to \UM} if
\w[.]{p\sb{M}\Phi=\phi} However, \w{\hIm(\Phi)\subseteq\UPM} need not be a
path space, so we need the following:

Let \w{\widetilde{\Phi}:C\bX\to \UM} denote the adjoint of $\Phi$
(see \wref{eqpathloop} and \cite[Ch.\ II, 1.3]{QuiH}), with
\w{\widetilde{\Phi}\circ\inc=\phi} for the inclusion \w[.]{\inc:\bX\hra C\bX}
Then \w{M':=\hIm(\phi)} is contained in \w[,]{M'':=\hIm(\widetilde{\Phi})}
fitting into a commutative diagram
$$
\xymatrix@R=13pt{
\bX \ar@/^{2.5pc}/[rrrrrrd]^{\phi}
\ar@{^{(}->}[d]^{\inc} \ar@{->>|}[rr]_{\psi} &&
\UM' \ar@{^{(}->}[rrd]_{j'} \ar@{^{(}->}[rrrrd]^{j} && \\
C\bX \ar@/_{2.0pc}/[rrrrrr]_{\widetilde{\Phi}} \ar@{->>|}[rrrr]^{\widetilde{\Psi}} &&&&
\UM'' \ar@{^{(}->}[rr]_{j''} && \UM
}
$$
\noindent and the original nullhomotopies fit into an adjoint
commutative diagram:
\mydiagram[\label{eqesnull}]{
&&&& \UPM'' \ar[d]^{p\sb{M''}} \ar@{^{(}->}[rr]_{Pj''} && \UPM \ar[d]^{p\sb{M}} \\
\bX \ar@/^{3.0pc}/[rrrrrru]\sp{\Phi} \ar@/_{1.5pc}/[rrrrrr]\sb{\phi}
\ar[rrrru]\sp{\Psi} \ar@{->>|}[rr]\sp<<<<<<<<<<<<<<{\psi} &&
\UM' \ar@{^{(}->}[rr]^{j'} && \UM'' \ar@{^{(}->}[rr]^{j''} && \UM~.
}
\noindent We write
\begin{myeq}\label{eqhhimage}
\hhIm(\Phi)~:=~P(\hIm(\widetilde{\Phi}))~,
\end{myeq}
\noindent and call a nullhomotopy $\Phi$ \emph{effectively surjective} if
\w{\UM=\UM''} in \wref{eqesnull} (that is, if \w[).]{\UPM=\hhIm(\Phi)}
We denote the set of all effectively surjective
nullhomotopies of $\phi$ by \w[.]{\hHom(\bX,PM)\sb{\phi}}
\end{defn}

\begin{defn}\label{dhtr}
For any \w{M\in\sMR} and effective epimorphism \w{\phi:\bX\efpic\UM}
we define \w{\hQ{\phi}} by the pullback square in \w[:]{\Sa}

\mydiagram[\label{eqmsmallestpb}]{
\hQ{\phi}\ar[rr]\sp<<<<<<<<<<<{(\hxi{\Phi})}\ar[d]\sb{\home{\phi}} &&
\prod'\sb{\textstyle{\mbox{\small $(j:M\hra M')\in\sMR$}}}\ \
\prod\sb{\textstyle{\mbox{\small $\Phi\in\hHom(\bX,PM')\sb{j\phi}$}}}
\ \UPM' \ar@<21ex>[d]\sb<<<<{\prod' p\sb{M'}}\\
\UM\ar[rr]\sp<<<<<<<<<<{(j)} &&
\prod'\sb{\textstyle{\mbox{\small $(j:M\hra M')\in\sMR$}}}\ \
\prod\sb{\textstyle{\mbox{\small $\Phi\in\hHom(\bX,PM')\sb{j\phi}$}}}\ \UM',
}
\noindent where \w{\stackrel{~'}{\prod}} indicates that empty factors are to
be omitted from the product (so the limit is in fact taken over a small diagram).

Note that if \w{j:M\hra C\UM\cong\UM'} is the inclusion into the cone,
\w{j\circ\phi} is nullhomotopic, so the upper right hand corner of
\wref[,]{eqmsmallestpb} and thus \w[,]{\hQ{\phi}} are never empty.

We now define \w{\cTR:\Sa\to\Sa} by
\begin{myeq}\label{eqmdsconst}
\cTR\bX~:=~\prod'\sb{M\in\sMR}~~
\prod\sb{\textstyle{\mbox{\scriptsize $\phi:\bX\hefpic M$}}}~\hQ{\phi}~.
\end{myeq}

Since there is only a \emph{set} of choices of $\UM$ and $\phi$ for which
\w[,]{\hQ{\phi}\neq\emptyset} the products in \wref{eqmdsconst} are in fact taken
over a set, rather than a proper class, of indices.
\end{defn}

\begin{lemma}\label{lfunctor}
The construction \w{\cTR} is functorial in \w[.]{f:\bX\to\bY}
\end{lemma}

\begin{proof}
Let \w{\psi:\bY\efpic\UM} be an effective epimorphism, \w{j:\UM\hra\UM'} an
inclusion in \w[,]{\sMR} and \w{\Psi:\bY\to\UPM'} an effectively
surjective nullhomotopy of \w[.]{j\circ\psi:\bY\to\UM'}

Set \w{\UM'':=\hIm(\psi\circ f)} and
\w[,]{\UM''':=\hhIm(\Psi\circ f)} as in \wref[,]{eqhhimage}
with \w{j'':\UM''\hra\UM} and \w{j':\UM'''\hra\UM'} the
inclusions. As in \wref{eqesnull} we also have an inclusion
\w[,]{j''':\UM''\hra\UM'''} fitting into the diagram:
\mydiagram[\label{eqcompmod}]{
\bY \ar@/^{3.5pc}/[rrrrrrrr]\sp{\Psi} \ar@{->>|}[rrrrrdd]\sb{\psi} &&
\bX \ar[ll]\sb<<<<<{f} \ar@{->>|}[rrd]\sb{\phi} \ar[rrrr]\sp{\Phi}
&&&& \UPM''' \ar[d]\sp{p\sb{M'''}} \ar@{^{(}->}[rr]\sp{Pj'}&&
\UPM' \ar[dd]\sp{p\sb{M'}} \\
&&&& \UM'' \ar@{^{(}->}[rd]\sp{j''} \ar@{^{(}->}[rr]\sp{j'''} &&
\UM''' \ar@{^{(}->}[rrd]\sp{j'} && \\
&&&&& M \ar@{^{(}->}[rrr]\sp{j} &&& \UM'
}
\noindent The map \w{\cTR f:\cTR\bX\to\cTR\bY} is defined into the factor
\w{\hQ{\psi}} of \w{\cTR\bY} by projecting from \w{\cTR\bX} onto
\w{\hQ{\phi}} and then onto the copy of \w{\UPM'''} indexed by $\Phi$,
which then maps by \w{Pj'} to the copy of \w{\UPM'} in \w{\hQ{\psi}} indexed by
$\Psi$. The copy of \w{\UM''} in the lower left corner of \wref{eqmsmallestpb}
for \w{\hQ{\phi}} maps by \w{j''} to the corresponding copy of $\UM$ for
\w[.]{\hQ{\psi}} The commutativity of the lower right hand parallelogram in
\wref{eqcompmod} ensures that this induces a well-defined map on the limits
\wref{eqmsmallestpb} and \wref[.]{eqmdsconst}
\end{proof}

\begin{defn}\label{dblambda}
For any \w{\bX\in\Sa} we define the cardinal \w{\lambda\sb{\bX}} to be
$$
\lambda\sb{\bX}~:=~
\sup\sb{M\in\sMR}\
\{|\hIm(\phi)|~|\ \phi:\bX\to \UM\}~\cup~\{|\hhIm(\Phi)|~|~\Phi:\bX\to \UPM\}~.
$$
\end{defn}

\begin{prop}\label{pstabilize}
For any \w{\bX\in\Sa} and \w{\lambda\geq\lambda\sb{\bX}} we have a canonical
isomorphism \w[.]{\cTR\bX\cong\hTR{\lambda}\bX}
\end{prop}

\begin{proof}
By construction, \w{\hTR{\lambda}\bX} is obtained
from the \dRma \w{\fX:=\rho\fMF\bX} by taking the limit over all maps
\w{\phi:\bX\to M\sb{0}} for \w{M\sb{0}\in(\sMR)\sb{\lambda}} of the diagram
described in \S \ref{rsopposite} (for various inclusions
\w[),]{j\sb{0}:M\sb{0}\hra M\sb{1}} where if the copy of \w{M\sb{0}} in the diagram is
indexed by \w[,]{\phi\sb{0}} then the two copies of \w{M\sb{1}} are indexed by
\w[.]{\phi\sb{1}:=j\sb{0}\circ\phi\sb{0}}

This limit therefore splits up into connected components, one for each
effective epimorphism \w[,]{\phi:\bX\efpic M\sb{0}} with all other simplicial
$R$-modules \w{M\sb{1}} in that component indexed by \w{j\circ\phi} for various
inclusions \w[.]{j:M\sb{0}\hra M\sb{1}} Thus \w{M\sb{0}} is initial among all
the simplicial modules \w{M\sb{1}} in its component of the diagram.

Moreover, the following diagram is cofinal in \wref[:]{eqcorncorn}
\mydiagram[\label{eqnewcorncorn}]{
\prod\sb{(\fX\lin{q\sb{M\sb{0}}})^{-1}(\phi)}PM\sb{0}
\ar[rd]\sb<<<<<<<<<<{\prod jp\sb{M\sb{0}}}
\ar[r]\sp{\prod Pj}&
\hspace*{3mm}\prod\sb{(\fX\lin{q\sb{M\sb{0}}})^{-1}(\phi)}PM\sb{1}
\ar[d]\sp<<<<{\prod p\sb{M\sb{1}}} &&
\prod\sb{(\fX\lin{q\sb{M\sb{1}}})^{-1}(j\phi)}PM\sb{1}
\ar[ll]\sb<<<<<{\prod(\fX\lin{j})\sp{\ast}}
\ar[d]\sp{\prod p\sb{M\sb{1}}} \\
M\sb{0} \ar[r]\sb<<<<<<<<<{(j)} \ar@/_{3.5pc}/[rrr]\sb{(j)} &
\prod\sb{(\fX\lin{q\sb{M\sb{0}}})^{-1}(\phi)}M\sb{1} &&
\prod\sb{(\fX\lin{q\sb{M\sb{1}}})^{-1}(j\phi)}M\sb{1}
\ar[ll]\sp{\prod(\fX\lin{j})\sp{\ast}}
}
\noindent Note that since \w{\phi:\bX\efpic M\sb{0}} is an effective epimorphism,
any nullhomotopy \w{\Phi:\bX\to PM\sb{0}} must be effectively surjective (see
\wref[),]{eqesnull} so we may replace the index set
\w{(\fX\lin{q\sb{M\sb{0}}})^{-1}(\phi)} by \w[.]{\hHom(\bX,PM\sb{0})\sb{\phi}}

Moreover, the map \w{\prod(\fX\lin{j})\sp{\ast}} into the lower middle product
(of copies of \w[,]{M\sb{1}} indexed again by \w[)]{\hHom(\bX,PM\sb{0})\sb{\phi}}
is simply the projection onto those factors of
\w{\prod\sb{\Psi\in(\fX\lin{q\sb{M\sb{1}}})^{-1}(j\phi)}M\sb{1}} indexed by
nullhomotopies \w{\Psi} for \w{j\circ\phi} which are \emph{not} in the image of
\w[.]{j\sb{\ast}:\Hom(\bX,PM\sb{0})\hra\Hom(\bX,PM\sb{1})}

Thus we may decompose this index set as a disjoint union
$$
(\fX\lin{q\sb{M\sb{1}}})^{-1}(j\phi)~=~
j\sb{\ast}\left((\fX\lin{q\sb{M\sb{0}}})^{-1}(\phi)\right)~\amalg~\New\sp{0}\sb{1}~,
$$
\noindent where \w{\New\sp{0}\sb{1}} consists of those nullhomotopies
\w{\Psi:\bX\to PM\sb{1}} of \w{j\circ\phi} which do not come from nullhomotopies
of $\phi$ itself.

The factors in the right hand products in \wref{eqnewcorncorn} indexed by
\w{\New\sp{0}\sb{1}} therefore fit into a cospan of the form

\mydiagram[\label{eqnewrightcorn}]{
M\sb{0} \ar[rr]\sp{(j)} && \prod\sb{\New\sp{0}\sb{1}}~M\sb{1}~&&
\prod\sb{\New\sp{0}\sb{1}}~PM\sb{1}~, \ar[ll]\sb{\prod p\sb{M\sb{1}}}
}
\noindent while the remainder of \wref{eqnewcorncorn} becomes
\mydiagram[\label{eqnewleftcorn}]{
M\sb{0} \ar[rr]\sp<<<<<<<<{\diag}&& \prod\sb{\hHom(\bX,PM\sb{0})\sb{\phi}}~M\sb{0}~&&
\prod\sb{\hHom(\bX,PM\sb{0})\sb{\phi}}~PM\sb{0}~, \ar[ll]\sb{\prod p\sb{M\sb{0}}}
}
\noindent where if \w{\hHom(\bX,PM\sb{0})\sb{\phi}} is empty (that is, $\phi$ itself
is not nullhomotopic), we replace \wref{eqnewleftcorn} by a single copy of
\w{M\sb{0}} mapping diagonally to \w[,]{\prod\sb{\New\sp{0}\sb{1}}~M\sb{1}}
by \S \ref{rsopposite}(b).

Note that any nullhomotopy \w{\Phi:\bX\to PM\sb{1}} of \w{j\circ\phi} which is not
effectively surjective in the sense of Definition \ref{dsrmod} necessarily factors
through \w{\bX~\xra{\Phi'}~PM\sb{2}~\xhra{Pj'}~PM\sb{1}} for some
\w{M\sb{0}\subseteq M\sb{2}\subseteq M\sb{1}} in \w[,]{(\sMR)\sb{\lambda}}
as in \wref[.]{eqesnull}

Therefore, if we replicate Diagram \wref{eqnewcorncorn} for the inclusion
\w[,]{j':M\sb{2}\hra M\sb{1}} again all such nullhomotopies $\Phi$ will be
omitted from the new set \w[,]{\New\sp{2}\sb{1}} so in the full diagram for
all of \w{(\sMR)\sb{\lambda}} the corner \wref{eqnewrightcorn} for \w{M\sb{1}}
is replaced by
\mydiagram[\label{eqnewrerightcorn}]{
M\sb{0} \ar[rr]\sp<<<<<<<<{(j)} && \prod\sb{\hHom(\bX,PM\sb{1})\sb{\phi}}~M\sb{1}~
&& \prod\sb{\hHom(\bX,PM\sb{1})\sb{\phi}}~PM\sb{1}~,\ar[ll]\sb{\prod p\sb{M\sb{1}}}
}
\noindent as in \wref[.]{eqmsmallestpb}

Thus whenever \w[,]{\lambda>\lambda\sb{\bX}} we have a cofinal diagram defining
\w{\hTR{\lambda}\bX} in which only those \w{M\in(\sMR)\sb{\lambda}} appear for
which there is either an effective epimorphism \w[,]{\bX\efpic M} or there is
an effectively surjective nullhomotopy \w{\bX\to\UPM} for some map
\w[:]{\bX\to M} that is, in fact we take only \w[,]{M\in(\sMR)\sb{\lambda\sb{\bX}}}
by Definition \ref{dblambda}.  This shows that in fact the natural map
\w{\cTR\bX\to\hTR{\lambda}\bX} is an isomorphism.
\end{proof}

\begin{notation}\label{dbiglambda}
For any commutative ring $R$ and \w[,]{\bY\in\Sa} let
$$
\widehat{\lambda}\sb{\bY}~:=~\sup\{\lambda\sb{\cTR\sp{n}\bY}~:\ n\in\NN\}~,
$$
\noindent let \w{\F\sb{\bY}} be the collection of $\Omega$-spectra whose
\ww{\Omega\sp{\infty}}-spaces are the objects of
\w[,]{(\sMR)\sb{\widehat{\lambda}\sb{\bY}}} and let \w{\bY\to\hWu} denote the
coaugmented cosimplicial space
\w{\bY\to\hWu\sb{\widehat{\lambda}\sb{\bY}}} of Corollary \ref{ccosimpres}.
\end{notation}

\begin{prop}\label{pcommon}
Given a commutative ring $R$, let \w{\G} denote the class of all simplicial
$R$-modules in \w[.]{\ho\Sa} For any \w[,]{\bY\in\Sa} \w{\bY\to\hWu}
is a weak $\G$-resolution of $\bY$ (cf.\ \cite[Definition 6.1]{BousC}).
\end{prop}

\begin{proof}
Given any simplicial $R$-module \w[,]{M\in\G} then as long as \w[,]{\kappa>|M|}
\w{[\bY,M]\leftarrow[\hWu\sb{\kappa},M]} is an acyclic augmented simplicial
abelian group, and each \w{\hW{n}\sb{\kappa}} is in  $\G$,
by Proposition \ref{pcosimpres}.  However, if
\w[,]{\kappa\geq\widehat{\lambda}\sb{\bY}} then \w{\hW{n}\sb{\kappa}}
is naturally isomorphic to \w{\cTR\sp{n+1}\bY}
for every \w[,]{n\geq 0} by Proposition \ref{pstabilize}, so in particular,
\w[.]{\hWu\cong\hWu\sb{\kappa}}
\end{proof}

\begin{corollary}\label{ccommon}
For $\G$ and \w{\bY\to\hWu} as above, \w{\uTot\hWu} (cf.\ \S \ref{dtotal})
is weakly equivalent to the $R$-completion \w{R\sb{\infty}\bY} of $\bY$.
\end{corollary}

\begin{proof}
The cosimplicial space \w{\hWu} is a weak $\G$-resolution of $\bY$, so in
particular it is a $\G$-complete expansion (cf.\ \cite[Definition 9.4]{BousC}),
and thus \w{\uTot\hWu\simeq\LG\bY} by \cite[Theorem 9.5]{BousC}, while
\w{\LG\bY\simeq R\sb{\infty}\bY} by \cite[\S 7.7]{BousC}.
\end{proof}

We thus obtain the following generalization of Theorem \ref{tresm}:

\begin{thm}\label{trecogms}
Let \w{\eA=(\uA{n})\sb{n=0}\sp{\infty}} be an $\Omega$-spectrum model of a
connective ring spectrum, with \w[.]{R:=\pi_{0}\eA} Assume that \w{\bY\in\Sa}
is $R$-good, and let $\fX$ be a \Fma structure on \w[,]{X=\map(\bY,\uA{0})}
for \w[.]{\F=\F\sb{\bY}} If \w{\hWu} is the cosimplicial space obtained from
$\fX$ as in \S \ref{dbiglambda}, then \w[.]{X\simeq\map(\uTot\hWu,\uA{0})}
\end{thm}

\begin{proof}
Let the class \w{\G'\subseteq\Obj\ho\Sa} consist of all spaces \w{\Omega\sp{\infty}N}
where $N$ is an $\eA$-module spectrum, and let \w{\G\subseteq\Obj\ho\Sa}
consist of all \w[,]{\Omega\sp{\infty}M} where $M$ is an \ww{\HR{R}}-module
spectrum. By \cite[IV, \S 2.7]{EKMMayR}, in fact all such spectra $M$ are
$R$-GEMs, so \w[.]{\G=\Obj\sMR}  Moreover, as
the proof of \cite[Theorem 9.7]{BousC}, \w{\G\subseteq\G'} and each
\ww{\G'}-injective object is $\G$-complete, so by \cite[Theorem 9.6]{BousC}
we have a natural weak equivalence \w{\LGp\bY\simeq\LG\bY} for every
\w[,]{\bY\in\Sa} and \w{\LG\bY\simeq R\sb{\infty}\bY}
by \cite[\S 7.7]{BousC}.

Therefore, \w{\uTot\hWu\simeq \LGp\bY} by Corollary \ref{ccommon}
and $\bY$ is $R$-good (that is, $\G$-good) if and only if it
is \ww{\G'}-good, by \cite[Definition 8.3]{BousC}.

Therefore, by \cite[Proposition 8.5]{BousC} \w{\bY\to\uTot\hWu} is a
\ww{\G'}-equivalence, so in particular \w{\map(\bY,\bA)\to\map(\uTot\hWu,\bA)}
is a weak equivalence.
\end{proof}

\begin{example}\label{egrgood}
Ring spectra as in Theorem \ref{trecogms} include \w{MU}
or \w[,]{BP} connective \w{ko} or \w[,]{ku} and of course the Eilenberg-Mac~Lane
spectrum \w{\HR{R}} for any commutative ring $R$.

Any simply-connected \w{\bY\in\Sa} is $R$-good.
If $R$ is a solid ring (cf.\ \cite{BKanC}) \wh e.g., \w{R\subseteq\QQ} or
\w{R=\Fp} \wwh and \w{\pi\sb{1}\bY} is $R$-perfect (that is,
\w[),]{(\pi\sb{1}\bY)\sb{\ab}\otimes R=0} then $\bY$ is $R$-good by
\cite[VII, 3.2]{BKanH}.

Note that for any commutative ring $R$ and \w[,]{\bY\in\Sa}
\w{(\core R)\sb{\infty}\bY\simeq R\sb{\infty}\bY} by \cite[I, Lemma 9.1]{BKanH},
where \w{\core R\subseteq R} is the maximal solid subring of $R$ (cf.\
\cite{BKanC}). Therefore, we can replace \w{R:=\pi\sb{0}\eA} by \w{\core R} in
the Theorem (and in the construction \w{\hWu} in \S \ref{dbiglambda}).
\end{example}

\begin{remark}\label{rrecover}
For such an $\eA$, Theorem \ref{trecogms}
provides us with a \emph{recovery} procedure for retrieving $\bY$ from the
\Fma structure \w{\fX=\fMF\bY} on \w[,]{X=\map(\bY,\bA)} for $\F$ as in
\S \ref{dsdiscrm}.
Of course, we cannot expect to recover $\bY$ up to weak equivalence, but
only to the extent that $\bY$ is determined by $X$ (namely, up to
\ww{\G'}-completion, for \w{\G'} as above).
However, it does not allow us to \emph{recognize} when an abstract \Fma
is in fact realizable.

Moreover, it implies that we can actually define a $\G$-completion for
an \emph{abstract} \Fma $\fX$: namely, \w{\LG\fX:=\uTot\hWu}
for \w{\hWu} as in \S \ref{dbiglambda}. However, without additional
assumptions, it need not be true that \w[.]{\fX\simeq\fMF(\LG\fX)}
\end{remark}

%
%
\sect{Realizing simplicial \Tal resolutions}
\label{crstr}

Our goal here is to show how a free simplicial (algebraic) resolution \w{\Vd}
of an enrichable \Tal $\Lambda$ can be realized in $\C$.  For this purpose,
we require the following:

\begin{defn}\label{dcwres}
For any $\fG$-sketch $\Theta$ (\S \ref{dfgsketch}), a \emph{CW-resolution} of a \Tal
$\Lambda$ is a cofibrant replacement \w{\vare:\Gd\xra{\simeq}c\Lambda}
(in the model category of simplicial \Tal[s] given by Proposition \ref{psimptal}),
equipped with a CW basis \w{(\oG{n})\sb{n=0}\sp{\infty}}
(as in \S \ref{dscwo}), with each \w{\oG{n}} a free \Tal[.]
\end{defn}

\begin{remark}
In fact, any CW object \w{\Gd} for which each \w{\oG{n}} is a free \Tal[,]
and each attaching map \w{\odz{G\sb{n}}} \wb{n\geq 0} surjects onto
\w[,]{Z\sb{n-1}\Gd} is a CW-resolution. Here we set \w{Z\sb{-1}\Gd:=\Lambda} and
\w[,]{\odz{G\sb{0}}:=\vare} so that
\begin{myeq}\label{eqaugmzero}
\vare\circ \odz{G\sb{1}}~=~0~.
\end{myeq}
\end{remark}

\begin{lemma}\label{lmoore}
Let \w{\Wu\in c\C} be a Reedy cofibrant cosimplicial object over a
model category $\C$ (as in \S \ref{snac}), and $\bB$ a homotopy group object
in $\C$. Then for any  Moore chain \w{\beta\in C\sb{n}[\Wu,\bB]} for the
simplicial group \w[:]{[\Wu,\bB]}
\begin{enumerate}
\renewcommand{\labelenumi}{(\alph{enumi})~}
\item $\beta$ can be realized by a map \w{b:\bW\sp{n}\to\bB} with
\w{b\circ d\sp{i}\sb{n-1}=0} for all \w[.]{1\leq i\leq n}
\item If $\beta$ is an algebraic Moore \emph{cycle}, we can choose a nullhomotopy
\w{H:\bW\sp{n-1}\to P\bB\subseteq \bB\sp{[0,1]}} for \w{b\circ \dz{n-1}} such
that \w{H\circ d\sp{j}\sb{n-2}=0} for \w[.]{1\leq j\leq n-1}
\end{enumerate}
\end{lemma}

\begin{proof}
Since \w{\Wu} is Reedy cofibrant, the simplicial space
\w{\Ud=\mapa(\Wu,\bB)\in s\Sa} is Reedy fibrant, so we have an isomorphism
\begin{myeq}\label{eqcommmoor}
\iota\sb{\star}~:~\pi\sb{i}C\sb{n}\Ud~\to~C\sb{n}\pi\sb{i}\Ud
\end{myeq}
\noindent (cf.\ \cite[X, 6.3]{BKanH}). Thus we can represent
\w{\alpha\in C\sb{n}\pi\sb{0}\Ud} by a map \w[,]{a\in C\sb{n}\Ud} which implies
(i)\vsm.

If $\alpha$ is a cycle, then \w{\partial\sb{n}(\alpha)=[a\circ \dz{n-1}]}
vanishes in \w[,]{\pi\sb{0}C\sb{n-1}\Ud} so we have a nullhomotopy $H$ for
\w{a\circ \dz{n-1}} in
$$
PC\sb{n-1}\mapa(\Wu,\bB)=C\sb{n-1}\mapa(\Wu,P\bB)~
\subseteq~\mapa(\bW\sp{n-1},P\bB)~,
$$
\noindent which implies (ii).
\end{proof}

The following result essentially dualizes (and extends) \cite[Theorem 3.16]{BlaCW}:

\begin{thm}\label{tres}
Assume given an enriched sketch $\bT$ in a model category $\C$ as
in \S \ref{snac}, with \w{\Theta:=\pi\sb{0}\bT} the associated algebraic
sketch, and let $\Lambda$ be a \Tal equipped with a CW-resolution \w[.]{\Vd}
If \w{\Lambda=\Lambda\sb{\fX}} is enriched by a \Tma $\fX$, then there is a
CW cosimplicial object \w{\Wu\in c\C} realizing \w{\Vd} \wwh that is, there
is an augmentation \w{\vn{\infty}:\fMT\bW\sp{0}\to\fX} such that
\w{\pi\sb{0}(\fMT\Wu)\to\pi\sb{0}\fX} is isomorphic to \w[.]{\Vd\to\Lambda}
If \w{\fX=\fMT\bY} for some \w[,]{\bY\in\C} then \w{\Vd\to\Lambda} can be
realized by a coaugmented cosimplicial object \w[.]{\bY\to\Wu}
\end{thm}

\begin{remark}\label{rlambda}
The cardinal $\lambda$ which bounds the size of the products in $\bT$ must be
chosen so that all the free \Tal[s] \w{\oV{n}} in the CW basis for \w{\Vd}
are represented by objects in $\bT$ (see \S \ref{dfralg}).
\end{remark}

\begin{proof}
We first choose once and for all objects \w{\uW{n}} in $\bT$ realizing
\w[,]{\oV{n}} in the sense of \S \ref{dfralg} \wh so
\w{\pi\sb{0}\fMT\uW{n}\cong\oV{n}} as \Tal[s.] This is possible by assumption
\ref{rlambda}.

We will construct the cosimplicial object \w{\Wu\in c\C} by a double induction:
in the outer induction, we construct a sequence of cosimplicial objects and maps:
\begin{myeq}\label{eqtower}
\dotsc~\to~\W{n}~\xra{\prn{n}}~\W{n-1}~\xra{\prn{n-1}}~\W{n-2}~\to~\dotsc~
\to~\W{0}~,
\end{myeq}
\noindent such that:

\begin{enumerate}
\renewcommand{\labelenumi}{(\alph{enumi})~}
\item \w{\Wu} is the dimensionwise limit of \wref[.]{eqtower}
\item Each cosimplicial object \w{\W{n}} is weakly $\G$-fibrant for
\w{\G:=\Obj\bT} (cf.\ \S \ref{dwgf}), as well as being cofibrant in the Reedy
model category (\S \ref{sremc}).
\item The simplicial \Tma \w{\fVnd{n}:=\fMT\W{n}} has an augmentation
\w[,]{\vn{n}:\fVn{n}{0}=\fMT\Wn{0}{n}\to\fX} which we can identify with a
$0$-simplex \w{\svn{n}} in \w{(\fX\lin{\Wn{0}{n}})\sb{0}} by Lemma \ref{lfreema}.
\item \w{\W{n}} is an $n$-coskeletal weak CW cosimplicial object
(cf.\ \S \ref{dccwo}), with CW basis \w{(\uW{k})\sb{k=0}\sp{n}}
(and zero CW basis object in dimensions $>n$).
\item The augmented simplicial \Tma \w{\fMT\W{n}} realizes \w{\Vd\to\Lambda}
through simplicial dimension $n$.
\item The maps \w{\prn{n}} restrict to a fibration weak equivalence
\w{\prnk{n}{k}:\Wn{k}{n}\to\Wn{k}{n-1}} for each \w[,]{0\leq k<n} so \w{\bW\sp{k}}
is the homotopy limit of the objects \w{\Wn{k}{n}} \wb[.]{n\geq 0}
\item The augmentation \w{\vn{n-1}:\fVnd{n-1}\to\fX} extends along the \Tma map
\w{\prn{n}\sp{\ast}:\fVnd{n-1}\to\fVnd{n}} to \w[\vsm.]{\vn{n}:\fVnd{n}\to\fX}
\end{enumerate}

\noindent\textbf{Step $\mathbf{0}$ of the outer induction\vsn.}

We start the induction with \w{\W{0}:=\cu{\uW{0}}} (the constant cosimplicial
object), which is both Reedy cofibrant and weakly $\G$-fibrant.
Note that because \w{\oV{0}} is a free \Tal[,] the \Tal augmentation
\w{\vare:\oV{0}\to\Lambda} corresponds to a unique element in
\w[,]{[\svn{0}]\in\Lambda\lin{\oV{0}}=\pi\sb{0}\fX\lin{\uW{0}}} for which we
may choose a representative \w[,]{\svn{0}\in\fX\lin{\uW{0}}\sb{0}} which
by Lemma \ref{lfreema} corresponds to a map of \Tma[s]
\w[\vsm.]{\vn{0}:\fVn{0}{0}=\fMT\uW{0}\to\fX}

\noindent\textbf{Step $\mathbf{1}$ of the outer induction\vsn.}

We choose a map \w{\udz{0}:\uW{0}\to\uW{1}} realizing the first attaching map
\w[,]{\odz{1}:\oV{1}\to V\sb{0}=\oV{0}} and define (the $1$-truncation of)
\w{\W{1}} by the diagram:
\mydiagram[\label{eqwone}]{
\Wn{0}{1} \ar@/_{1.5pc}/[d]\sb{\dz{0}} \ar@/^{1.5pc}/[d]\sp{d\sp{1}\sb{0}} &=&
\uW{0} \ar@/_{0.5pc}/[d]\sb{\dz{0}=d\sp{1}\sb{0}=\Id}
\ar[drr]\sp{\udz{0}} && \times &&
P\uW{1} \ar[dll]\sb{d\sp{1}\sb{0}=p} \ar@/^{0.5pc}/[d]\sp{\dz{0}=d\sp{1}\sb{0}=\Id} \\
\Wn{1}{1} \ar[u]\sp{s\sp{0}} &=& \uW{0}\ar@/_{0.5pc}/[u]\sb{=} &\times &
\uW{1} & \times & P\uW{1}\ar@/^{0.5pc}/[u]\sp{=}
}
\noindent Here \w{p:P\bX\to\bX} is the path fibration in $\C$, defined
as in \wref[.]{eqpathloop}

To define the augmentation \w{\svn{1}} as a $0$-simplex in
\w{(\fX\lin{\Wn{0}{1}})\sb{0}} extending
\w{\svn{0}\in\fX\lin{\Wn{0}{0}}=\fX\lin{\uW{0}}} (see (b) above), we use the
fact that
$$
\fX\lin{\Wn{0}{1}}~=~\fX\lin{\uW{0}\times P\uW{1}}~=~
\fX\lin{\uW{0}}~\times~\fX\lin{P\uW{1}}~=~
\fX\lin{\uW{0}}~\times~P\fX\lin{\uW{1}}~,
$$
\noindent by \S \ref{dma}(a)-(b), so we need only to find a $0$-simplex $H$
in \w{P\fX\lin{\uW{1}}} \wwh which, by \wref[,]{eqpathloop}
is a $1$-simplex in \w{\fX\lin{\uW{1}}} with \w[.]{d\sb{1}H=0}

In order to qualify as an augmentation \w{\fVnd{1}\to\fX} of simplicial
\Tma[s,]  \w{\vn{1}} must satisfy the simplicial identity
\begin{myeq}\label{eqaugid}
\vn{1}\circ d\sb{0}~=~\vn{1}\circ d\sb{1}~:~\fVn{1}{1}~\to~\fX
\end{myeq}
\noindent as maps of \Tma[s] \wwh or equivalently, these must correspond to
the same $0$-simplex in
$$
\fX\lin{\Wn{1}{1}}~=~\fX\lin{\uW{0}\times \uW{1}\times \uW{1}}~=~
\fX\lin{\uW{0}}\times\fX\lin{\uW{1}}\times\fX\lin{\uW{1}}~.
$$
In the first factor and third factor this obviously holds, so we need only
consider the two $0$-simplices \w[:]{\fX\lin{\uW{1}}} in other words,
since the path fibration $p$ in \wref{eqwone} becomes \w{d\sb{0}} in
the simplicial set \w{\fX\lin{\uW{1}}} (since it is induced by an inclusion
\w[),]{\Delta[0]\hra\Delta[1]} we must choose the nullhomotopy
$H$ so that \w{d\sb{0}H} is the $0$-simplex \w[.]{(\udz{0})\sb{\#}\svn{0}}

But by \wref{eqaugmzero} we know that \w{\vare\circ\odz{0}=0} in \w[,]{\TAlg}
which implies (by our choices of \w{\udz{0}} and \w{\svn{0}} representing
\w{\odz{0}} and $\vare$, respectively) that \w{(\udz{0})\sb{\#}\svn{0}}
is nullhomotopic, so we can choose an $H$ as required\vsm.

\noindent\textbf{Step $\mathbf{n}$ of the outer induction \wb[\vsn:]{n\geq 2}}

Assume given \w{\W{n-1}} satisfying (a)-(g) above, we construct an
intermediate $n$-coskeletal restricted cosimplicial object \w{\tWn{\bullet}{n}}
(cf.\ \S \ref{sscso}) by a descending induction on the cosimplicial dimension
\w[.]{0\leq k\leq n}

We require that, for each \w[,]{0\leq k<n} the object \w{\tWn{k}{n}\in\C}
is defined by the (homotopy) pullback diagram:
\mydiagram[\label{eqdopb}]{
\ar @{} [drr] |<<<<<{\framebox{\scriptsize{PB}}}
\tWn{k}{n} \ar@{->>}[d]\sb{\psn{k}{n}}\sp{\simeq} \ar[rr]\sp{\qk{k}} &&
(\Omega\sp{n-k-1}\uW{n})\sp{\Delta[1]} \ar@{->>}[d]\sp{\simeq}\sb{\ev\sb{0}} \\
\Wn{k}{n-1} \ar[rr]\sp{\etk{k}} && \Omega\sp{n-k-1}\uW{n}
}
\noindent for some map \w{\etk{k}} such that
\begin{myeq}\label{eqetak}
\etk{k}\circ d\sp{i}\sb{k-1}=0\hsp \text{for all}\hsm 1\leq i\leq k~.
\end{myeq}

The idea is that the projection of the coface map
\w{\vdz{k-1}:\Wn{k-1}{n-1}\to\tWn{k}{n}} onto \w{\Omega\sp{n-k-1}\uW{n}}
in \wref{eqdopb} describes the value \w{\ak{k-1}} of a certain ``universal
\wwb{n-k-1}-th order cohomology operation'' (see \S \ref{scrho} below),
while the projection onto \w{(\Omega\sp{n-k-1}\uW{n})\sp{\Delta[1]}}
describes a homotopy \w{\Fk{k-1}} between this value and the corresponding
element \w{\etk{k}\circ \dz{k-1}} in the cohomology of
\w[.]{\Wn{k-1}{n-1}} This higher order operation is defined by
composing the homotopy of a lower order operation with
some map (which we think of as a primary cohomology operation) \wh in this
case, \w[\vsm .]{\Fk{k}\circ \dz{k-1}}

At the $k$-th stage, we assume that we have defined \w{\tWn{i}{n}} for
\w[,]{n\geq i\geq k+1} with all coface maps
\w{d\sp{j}\sb{i}:\tWn{i}{n}\to\tWn{i+1}{n}} for all $j$ and
\w[,]{n\geq i\geq k+1} as well as \w{\vdz{k}:\Wn{k}{n-1}\to\tWn{k+1}{n}}
(if \w[),]{k<n} with the $0$-th coface map of \w{\tWn{\bullet}{n}} given by:
\begin{myeq}\label{eqtdz}
\tdz{k}~:=~\vdz{k}\circ\psn{k}{n}~:~\tWn{k}{n}~\to~\tWn{k+1}{n}~.
\end{myeq}
\noindent and \w{\dz{k}:\Wn{k}{n-1}\to\Wn{k+1}{n-1}} equal to
\w[.]{\psn{k}{n}\circ\vdz{k}}

We write:
\begin{myeq}\label{eqhk}
\Fk{k}~:=~\qk{k+1}\circ \vdz{k}:\Wn{k}{n-1}~\to~
(\Omega\sp{n-k-2}\uW{n})\sp{\Delta[1]}
\end{myeq}
\noindent for the homotopy given in the previous stage, so the map \w{\vdz{k}}
into the pullback \w{\tWn{k+1}{n}} in \wref{eqdopb} is determined by the
two compatible maps \w{\Fk{k}} and \w[.]{\dz{k}} We also set:
\begin{myeq}\label{eqak}
\ak{k}~:=~ev\sb{1}\circ \Fk{k}:\Wn{k}{n-1}~\to~\Omega\sp{n-k-2}\uW{n}
\end{myeq}
\noindent for the previous value of the corresponding higher order operation,
so
\begin{myeq}\label{eqhtpyhk}
\Fk{k}:\etk{k+1}\circ \dz{k}~\sim~\ak{k}~.
\end{myeq}

We also assume by induction that:
\begin{myeq}\label{eqhvanish}
\Fk{k}\circ d\sp{i}\sb{k-1}~=~0\hsm \text{for all}\hsm 1\leq i\leq k~.
\end{myeq}
\noindent This implies that \w{\ak{k}\circ d\sp{i}\sb{k-1}=0} for
\w[,]{i\geq 1} but in fact we require that:
\begin{myeq}\label{eqavanish}
\ak{k}\circ d\sp{i}\sb{k-1}~=~0\hsm \text{for all}\hsm 0\leq i\leq k~\vsm.
\end{myeq}

\noindent\textbf{Step $\mathbf{k=n}$ of the descending induction\vsn:}

We start the induction by setting
\begin{myeq}\label{eqninduct}
\tWn{n}{n}~:=~\Wn{n}{n-1}\times\uW{n} ~.
\end{myeq}

By assumption \w{\W{n-1}} is Reedy cofibrant, so the bisimplicial set
$$
\Ud~:=~\map_{\C}(\W{n-1},\uW{n})
$$
\noindent is Reedy fibrant. Moreover, since
\w{\pi\sb{0}\fMT\Wn{k}{n-1}\cong V\sb{k}} for all \w{0\leq k<n} by (c) above,
the algebraic attaching map \w{\odz{n}:\oV{n}\to V\sb{n-1}} is a homotopy
class
\begin{myeq}\label{eqmoorecyc}
\alpha\in\pi\sb{0}\bU\sb{n-1}~=~[\Wn{n-1}{n-1},\,\uW{n}]
~=~\pi\sb{0}(\fMT\Wn{n-1}{n-1}\lin{\uW{n}})
~=~V\sb{n-1}\lin{\uW{n}}~,
\end{myeq}
\noindent where the last equality follows from Lemma \ref{lfreeta}.
This $\alpha$ is a Moore chain in \w{\pi\sb{0}\Ud} by Definition \ref{dscwo},
so by Lemma \ref{lmoore}(a), \w{\odz{n}} can be represented by a
continuous map \w{\udz{n-1}:\Wn{n-1}{n-1}\to\uW{n}} satisfying:
\begin{myeq}\label{eqvnish}
\udz{n-1}\circ d\sp{j}\sb{n-2}=0\hsm \text{for all}\hsm 1\leq j\leq n-1~.
\end{myeq}
\noindent This defines \w{\vdz{n-1}:\Wn{n-1}{n-1}\to\tWn{n}{n}} into the product
\wref[,]{eqninduct} extending the given face map
\w[\vsm.]{\dz{n-1}:\Wn{n-1}{n-1}\to\Wn{n}{n-1}}

\noindent\textbf{Step $\mathbf{k=n-1}$ of the descending induction\vsn:}

We define \w{\tWn{n-1}{n}} by the pullback diagram \wref[,]{eqdopb} with
\w[.]{\etk{n-1}=0} Thus:
\begin{myeq}\label{eqnminduct}
\tWn{n-1}{n}~:=~\Wn{n-1}{n-1}~\times~P\uW{n}~.
\end{myeq}

By \wref[,]{eqattach} the class $\alpha$ of \wref{eqmoorecyc} is in fact a
Moore cycle, so by Lemma \ref{lmoore}(b) we have a nullhomotopy
\begin{myeq}\label{eqnullh}
\Fk{n-2}:\udz{n-1}\circ \dz{n-2}\sim 0
\end{myeq}
\noindent satisfying \wref[.]{eqhvanish}

We define \w{\vdz{n-2}:\Wn{n-2}{n-1}\to\tWn{n-1}{n}} into the new factor
\w{P\uW{n}} (and extending the given face map
\w[)]{\dz{n-2}:\Wn{n-2}{n-1}\to\Wn{n-1}{n-1}} to be
\begin{myeq}\label{eqnmszero}
\Fk{n-2}:\Wn{n-2}{n-1}~\to~P\uW{n}~.
\end{myeq}

We define the coface map \w{\td\sp{1}\sb{n-1}:\tWn{n-1}{n}\to\tWn{n}{n}}
by the given \w{d\sp{1}\sb{n-1}:\Wn{n-1}{n-1}\to\Wn{n}{n-1}} into the first
factor of \wref[,]{eqninduct} and the composite
\begin{myeq}\label{eqndone}
\tWn{n-1}{n-1}~\xra{\proj\sb{P\uW{n}}}~P\uW{n}~\xra{p}~\uW{n}~,
\end{myeq}
\noindent onto the second factor of \wref{eqninduct} (where
\w{p:P\uW{n}\to\uW{n}} is the path fibration).

The remaining face maps
\w{\td\sp{i}\sb{n-1}:\tWn{n-1}{n}\to\tWn{n}{n}} \wb{2\leq i\leq n}
extend the given \w{d\sp{i}\sb{n-1}:\Wn{n-1}{n-1}\to\Wn{n}{n-1}} by the zero
map into the CW basis \w[.]{\uW{n}}

The only cosimplicial identity that can be verified at this stage is
$$
\td\sp{1}\sb{n-1}\td\sp{0}\sb{n-2}~=~\td\sp{0}\sb{n-1}\td\sp{0}\sb{n-2}~,
$$
\noindent which follows from the fact that
\w[,]{p\circ \Fk{n-2}=\udz{n-1}\circ \dz{n-2}} by \wref[\vsm .]{eqnullh}

\noindent\textbf{Step $\mathbf{k}$ of the descending induction
\wb[\vsn:]{0<k\leq n-2}}

We define the map \w{\etk{k}:\Wn{k}{n-1}\to\Omega\sp{n-k-1}\uW{n}} as follows:

By assumption we are given a homotopy
\w[,]{\Fk{k}:\etk{k+1}\circ\dz{k}\sim\ak{k}}
so we get a homotopy:
\begin{myeq}\label{eqselfnull}
\Fk{k}\circ \dz{k-1}:\etk{k+1}\circ\dz{k}\circ \dz{k-1}~\sim~\ak{k}\circ
\dz{k-1}
\end{myeq}
\noindent where
\w{\etk{k+1}\circ\dz{k}\circ\dz{k-1}:\Wn{k-1}{n-1}\to\Omega\sp{n-k-2}\uW{n}}
is the zero map by \wref[,]{eqetak} Definition \ref{dccwo}(c), and the identity
\w[.]{d\sp{1}d\sp{0}=d\sp{0}d\sp{0}} Since also \w{\ak{k}\circ \dz{k-1}=0}
by \wref[,]{eqavanish}
\w{\Fk{k}\circ \dz{k-1}:\Wn{k-1}{n-1}\to(\Omega\sp{n-k-2}\uW{n})\sp{\Delta[1]}}
is a self-nullhomotopy, so it factors through the inclusion
\w{i\sb{k}:\Omega\sp{n-k-1}\uW{n}\hra(\Omega\sp{n-k-2}\uW{n})\sp{\Delta[1]}}
and thus defines a map \w{\ak{k-1}:\Wn{k-1}{n-1}\to\Omega\sp{n-k-1}\uW{n}}
with
\begin{myeq}\label{eqikak}
i\sb{k}\circ\ak{k-1}~=~\Fk{k}\circ \dz{k-1}~.
\end{myeq}
\noindent Moreover,
$$
i_{k}\circ\ak{k-1}\circ d\sp{i}\sb{k-2}~=~\Fk{k}\circ \dz{k-1}\circ
d\sp{i}\sb{k-2}~=~\Fk{k}\circ d\sp{i+1}\sb{k-1}\circ \dz{k-2}~=~0
$$
\noindent for all \w{0\leq i\leq k-1} by \wref[,]{eqhvanish} so \wref{eqavanish}
holds for \w{\ak{k-1}} since \w{i\sb{k}} is monic.

Therefore, \w{\ak{k-1}} is a \wwb{k-1}cycle for the Reedy fibrant bisimplicial
set \w[,]{\Ud:=\map\sb{\C}(\W{n-1},\,\Omega\sp{n-k-1}\uW{n})} and thus in
particular represents a \wwb{k-1}cycle \w{[\ak{k-1}]} for
\w[,]{\Vd\lin{\Omega\sp{n-k-1}\uW{n}}} as in \wref[.]{eqmoorecyc}

Because \w{\Vd\to\Lambda} is a resolution, and thus acyclic, there is a class
\w{\gamma\sb{k}\in\oV{k}\lin{\Omega\sp{n-k-1}\uW{n}}} with
\w[.]{\odz{k}(\gamma\sb{k})~=~[\ak{k-1}]} Moreover, the map
\w{\ophi{k}{n-1}:\Wn{k}{n-1}\to\uW{k}} of \wref{eqcwstru} induces the inclusion
\w[,]{(\ophi{k}{n-1})\sp{\ast}:\oV{k}\hra V_{k}} so we have a class
\w{[\etk{k}]:=(\ophi{k}{n-1})\sp{\ast}(\gamma\sb{k})\in
V\sb{k}\lin{\Omega\sp{n-k-1}\uW{n}}} which is a Moore chain by \S \ref{dccwo}(b).

By Lemma \ref{lmoore}(a) we can represent this class by a map
\w{\etk{k}:\Wn{k}{n-1}\to\Omega\sp{n-k-1}\uW{n}} satisfying \wref[,]{eqetak}
while by Lemma \ref{lmoore}(b) we have a homotopy
\begin{myeq}\label{eqhkmone}
\Fk{k-1}:\etk{k}\circ \dz{k-1}\sim\ak{k-1}:\Wn{k-1}{n-1}\to
\Omega\sp{n-k-1}\uW{n}
\end{myeq}
\noindent satisfying \wref{eqhvanish} for \w[.]{k-1}

We now define \w{\tWn{k}{n}} by the pullback diagram \wref[,]{eqdopb} in which
both vertical arrows are weak equivalences. To define the coface map
\w{\vdz{k-1}:\Wn{k-1}{n-1}\to\tWn{k}{n}} extending
\w[,]{\dz{k-1}:\Wn{k-1}{n-1}\to\Wn{k}{n-1}} we set
\w{\qk{k}\circ\vdz{k-1}:\Wn{k-1}{n-1}\to(\Omega\sp{n-k-1}\uW{n})\sp{\Delta[1]}}
equal to \w[.]{\Fk{k-1}}

To define the face map \w{\td\sp{1}\sb{k}:\tWn{k}{n}\to\tWn{k+1}{n}} into the
pullback (extending the given map
\w[),]{d\sp{1}\sb{k}\circ\psn{k}{n}:\tWn{k}{n-1}\to\Wn{k+1}{n-1}}
it suffices to specify the composite
$$
\qk{k+1}\circ\td\sp{1}\sb{k}:\tWn{k}{n}~\to~(\Omega\sp{n-k-2}\uW{n})\sp{\Delta[1]}~,
$$
\noindent which we set equal to the composite:
\begin{myeq}\label{eqkdone}
\tWn{k}{n}~\xra{\qk{k}}~(\Omega\sp{n-k-1}\uW{n})\sp{\Delta[1]}~
\xra{\ev\sb{1}}~\Omega\sp{n-k-1}\uW{n}~\xra{i\sb{k}}~
(\Omega\sp{n-k-2}\uW{n})\sp{\Delta[1]}~.
\end{myeq}
\noindent This indeed defines a map into the pullback
\wref{eqdopb} for \w[,]{k+1} since the composite
$$
\Omega\sp{n-k-1}\uW{n}~\xra{i\sb{k}}~(\Omega\sp{n-k-2}\uW{n})\sp{\Delta[1]}~
\xra{\ev\sb{0}}~\Omega\sp{n-k-2}\uW{n}
$$
\noindent is zero, which matches \wref[.]{eqetak}

The remaining face maps \w{\td\sp{i}\sb{k}:\tWn{k}{n}\to\tWn{k+1}{n}}
\wb{2\leq i\leq k+1} are defined by extending the given
\w{d\sp{i}\sb{k}:\Wn{k}{n-1}\to\Wn{k+1}{n-1}} by the zero map into
\w{(\Omega\sp{n-k-2}\uW{n})\sp{\Delta[1]}} (which again matches
\wref[).]{eqetak}

To verify the cosimplicial identity
\begin{myeq}\label{eqzeroone}
\td\sp{1}\sb{k}\circ\vdz{k-1}~=~\tdz{k}\circ\vdz{k-1}~:~
\Wn{k-1}{n}~\to~\tWn{k+1}{n}~,
\end{myeq}
\noindent it suffices to check the post-composition with \w[,]{\qk{k+1}} where:
\begin{equation*}
\begin{split}
\qk{k+1}\circ d\sp{1}\sb{k}\circ\vdz{k-1}~=&~
i\sb{k}\circ\ev\sb{1}\circ \qk{k}\circ\vdz{k-1}~=~
i\sb{k}\circ \ev\sb{1}\circ \Fk{k-1}\\
~=&~i\sb{k}\circ\ak{k-1}~=~
\Fk{k}\circ\dz{k-1}~=~\qk{k+1}\circ\vdz{k}\circ\dz{k-1}~=~
\qk{k+1}\circ \dz{k}\circ\vdz{k-1}~,
\end{split}
\end{equation*}
\noindent by \wref[,]{eqkdone} \wref{eqhkmone} \wref[,]{eqikak} and
\wref[.]{eqtdz}

The identities \w{\td\sp{j}\sb{k+1}\td\sp{i}\sb{k}=
\td\sp{i}\sb{k+1}\td\sp{j-1}\sb{k}}
hold trivially, with both sides vanishing, for all for \w{k+2\geq j>i\geq 0}
except for \w[.]{(j,i)\in\{(1,0),(2,0),(2,1)\}} We check these three cases:

\begin{enumerate}
\renewcommand{\labelenumi}{(\alph{enumi})~}
\item The case \w{\td\sp{1}\sb{k+1}\tdz{k}=\tdz{k+1}\tdz{k}} is \wref[,]{eqzeroone}
which was already verified in step \w[.]{k+1}
\item To show
\w[,]{\td\sp{2}\sb{k+1}\td\sp{1}\sb{k}=\td\sp{1}\sb{k+1}\td\sp{1}\sb{k}}
it suffices to check the post-composition with \w[,]{\qk{k+2}} where
\w{\qk{k+2}\circ\td\sp{2}\sb{k+1}=0} by definition, while
$$
\qk{k+2}\circ\td\sp{1}\sb{k+1}\circ\td\sp{1}\sb{k}~=~
i\sb{k+1}\circ \ev\sb{1}\circ \qk{k+1}\circ\td\sp{1}\sb{k}~=~
i\sb{k+1}\circ \ev\sb{1}\circ i\sb{k}\circ \ev\sb{1}\circ \qk{k}~=~0
$$
\noindent since
\w{\ev\sb{1}\circ i\sb{k}:\Omega\sp{n-k-1}\uW{n}\to\Omega\sp{n-k-2}\uW{n}}
is the zero map.
\item To show \w[,]{\td\sp{2}\sb{k+1}\tdz{k}=\tdz{k+1}\td\sp{1}\sb{k}}
it suffices to check the post-composition with \w[,]{\qk{k+2}} where again
\w{\qk{k+2}\circ\td\sp{2}\sb{k+1}=0} by definition, while
$$
\qk{k+2}\circ\tdz{k+1}\circ\td\sp{1}\sb{k}~=~
\Fk{k+1}\circ \psn{k+1}{n}\circ\td\sp{1}\sb{k}~=~
\Fk{k+1}\circ d\sp{1}\sb{k}\circ \psn{k}{n}~=~0
$$
\noindent (in the notation of \wref[),]{eqdopb} by \wref[\vsm .]{eqhvanish}
\end{enumerate}

\noindent\textbf{Step $\mathbf{k=0}$ of the descending induction\vsn:}

If \w[,]{\fX=\fMT\bY} the last step of the induction is no different from the
general $k$, with \w[.]{\Wn{-1}{n-1}:=\bY} However, in the general case we
no longer have an object \w{\Wn{k-1}{n-1}} in $\C$ for \w[,]{k=0} so we
must modify our construction somewhat:

By induction the homotopy \w{\Fk{0}} is a $1$-simplex in
\w[,]{(\fMT\Wn{0}{n-1}\lin{\Omega\sp{n-2}\uW{n}})\sb{1}} (for which
\wref{eqhvanish} is vacuous).
Its simplicial face maps are \w{d\sb{0}\Fk{0}=\etk{1}\circ\dz{0}} and
\w[,]{d\sb{1}\Fk{0}=\ak{0}} respectively.

Applying \w{\vn{n-1}:\fMT\Wn{0}{n-1}\to\fX} to
\w{\Fk{0}} yields a $1$-simplex
\w{\vn{n-1}(\Fk{0})\in\fX\lin{\Omega\sp{n-2}\uW{n}}} with
\begin{equation*}
\begin{split}
d\sb{0}\vn{n-1}(\Fk{0})~=&~\vn{n-1}(\etk{1}\circ\dz{0})~=~
(\vn{n-1}\circ \dz{0})\sb{\#}(\etk{1})
~=~(\vn{n-1}\circ d\sp{1}\sb{0})\sb{\#}(\etk{1}\circ\ophi{0}{n-1})\\
~=&~\vn{n-1}(\etk{1}\circ d\sp{1}\sb{0})~=~0
\end{split}
\end{equation*}
\noindent by \wref{eqaugid} and \wref[.]{eqetak}

Similarly, since
\w{i\sb{1}\circ\ak{0}=\Fk{1}\circ\dz{0}:
\Wn{0}{n-1}\to(\Omega\sp{n-3}\uW{n})\sp{\Delta[1]}} for
\w[,]{i\sb{1}:\Omega\sp{n-2}\uW{n}\hra (\Omega\sp{n-3}\uW{n})\sp{\Delta[1]}}
we see that:
\begin{equation*}
\begin{split}
(i\sb{1})\sb{\#}d\sb{1}\vn{n-1}(\Fk{0})~=&~
d\sb{1}\vn{n-1}(i\sb{1}\circ \Fk{0})~=~
\vn{n-1}(i\sb{1}\circ\ak{0})~=~
\vn{n-1}(\Fk{1}\circ\dz{0})\\
~=&~(\vn{n-1}\circ\dz{0})\sb{\#}(\Fk{1})
~=~(\vn{n-1}\circ d\sp{1}\sb{0})\sb{\#}(\Fk{1})~=~
\vn{n-1}(\Fk{1}\circ d\sp{1}\sb{0})~=~0
\end{split}
\end{equation*}
\noindent by \wref[.]{eqhvanish}

Since \w{i\sb{1}} is a cofibration in $\C$ by \wref{eqpathloop} (and the fact that
any \w{\uW{n}\in\bT} is cofibrant), by \S \ref{dma}(c)
\w{(i\sb{1})\sb{\#}} is monic, so \w[,]{d\sb{1}\vn{n-1}(\Fk{0})=0}  too.

Thus the $1$-simplex \w{\vn{n-1}(\Fk{0})\in\fX\lin{\Omega\sp{n-2}\uW{n}}}
actually defines a $0$-simplex \w{\ak{-1}} in
\w[,]{\fX\lin{\Omega\sp{n-1}\uW{n}}} representing some class
\w[.]{[\ak{-1}]\in\pi\sb{0}\fX\lin{\Omega\sp{n-1}\uW{n}}=
\Lambda\lin{\Omega\sp{n-1}\uW{n}}}

Because \w{\Vd} is a \Tal resolution of $\Lambda$, \w{\vare:V\sb{0}\to\Lambda}
is surjective, so there is a class
\w{[\etk{0}]\in V\sb{0}\lin{\Omega\sp{n-1}\uW{n}}} with
\w[,]{\vare([\etk{0}])=[\ak{-1}]}  which we can represent by a map
\w[,]{\etk{0}:\Wn{0}{n-1}\to\Omega\sp{n-1}\uW{n}} together with a $1$-simplex
\w{\Fk{-1}\in\fX\lin{\Omega\sp{n-1}\uW{n}}} with
\w{d\sb{0}\Fk{-1}=\vn{n-1}(\etk{0})} and \w[.]{d\sb{0}\Fk{-1}=\ak{-1}}

We can thus use \w{\etk{0}} to define \w{\tWn{0}{n}}
by the pullback diagram \wref[,]{eqdopb}  and use \w{\Fk{-1}} to extend
\w{\vn{n-1}} to \w[\vsm.]{\vn{n}:\fMT\tWn{0}{n}\to\fX}

\noindent\textbf{Completing the $n$-th step of the outer induction\vsn:}

Once the inner (descending) induction is completes, we define a full
cosimplicial object \w{\vWu{n}} by ascending induction on the cosimplicial
dimension $k$. This will be equipped with a map of restricted cosimplicial
objects \w[,]{g:\vWu{n}\to\tWn{\bullet}{n}} such that the composite
map \w{f:=\psn{}{n}\circ g:\vWu{n}\to\W{n-1}} is a map of cosimplicial objects
which is a dimensionwise weak equivalence\vsm .

We start with \w{\vWn{0}{n}:=\tWn{0}{n}} and \w[,]{g\sp{0}:=\Id} and define
\w{\vWn{k}{n}} by the pullback diagram:
\mydiagram[\label{eqmatchpb}]{
\ar @{} [drrrr] |<<<<<<<<{\framebox{\scriptsize{PB}}}
\vWn{k}{n} \ar[d]\sp{g\sp{k}} \ar[rrrr]\sp{\zn{k}{n}} &&&&
M\sp{k}\vWu{n} \ar[d]\sp{M\sp{k}f}\\
\tWn{k}{n} \ar[rr]\sp{\psn{k}{n}} && \Wn{k}{n-1} \ar[rr]\sp{\zn{k}{n-1}}
&& M\sp{k}\W{n-1}
}
\noindent (using the notation of \S \ref{slmo}).

The codegeneracy maps of \w{\vWu{n}} are defined by \wref[,]{equnivcod}
while the coface map
\w{\hat{d}\sp{j}\sb{k-1}:\vWn{k-1}{n}\to\vWn{k}{n}} is defined by the maps
\w{\td\sp{j}\sb{k-1}:\tWn{k-1}{n}\to\tWn{k}{n}} constructed in the descending
induction and the induced map \w{M\sp{k-1}\vWu{n}\to M\sp{k}\vWu{n}} determined
by the current induction, \wref[,]{equnivcod} and the cosimplicial
identities\vsn.

Finally, we let \w{h:\W{n}~\xepic{\simeq}~\vWu{n}} be any Reedy cofibrant
replacement (cf. \S \ref{sremc}). Note that \w{\W{n}} is still a weak CW
object, with CW basis \w[,]{(\uW{k})\sb{k=0}\sp{n}} and the map
\w{\oph\sp{k}\sb{[n]}:\Wn{k}{n}\to\uW{k}} of \wref{eqcwstru} defined to be the
composite:
$$
\Wn{k}{n}~\xra{h\sp{k}}~\vWn{k}{n}~\xra{g\sp{k}}~
\tWn{k}{n}~\xra{\psi\sp{k}\sb{[n]}}~\Wn{k}{n-1}~\xra{\oph\sp{k}\sb{[n-1]}}~\uW{k}~.
$$
\noindent Since $h$, $g$, and \w{\psn{}{n}} are maps of (restricted)
cosimplicial objects, by construction (and \wref[)]{eqhvanish} we see that
\wref{eqvantwo} is indeed satisfied.

This completes the $n$-th step of the outer induction, and thus the proof of
the Theorem.
\end{proof}

\begin{mysubsection}{Cosimplicial resolutions and higher operations}
\label{scrho}
The particular construction used to produce the realization \w{\Wu} of
the algebraic resolution \w{\Vd\to\Lambda} actually encodes, in an explicit
form, the additional higher order information that is needed to distinguish
between different objects $\bY$ realizing the given \Tal $\Lambda$.

For example, if \w{\C=\Sa} and \w[,]{\eA=\HR{R}} this
additional data takes the form of the higher order cohomology operations (see, e.g.,
\cite{AdamsN,BMarkH,MaunC,GWalkL}), as follows:

Note that the $n$-th object \w{\oV{n}} of the algebraic CW basis for
\w{\Vd} is a coproduct of free monogenic \TRal[s] of the form
\w[,]{H\sp{\ast}(\KR{n_{i}};R)} each of which is indexed
by a map \w{\phi_{i}:H\sp{\ast}(\KR{n_{i}};R)\to V_{n-1}} (the restriction of
the attaching map \w[).]{\odz{n}:\oV{n}\to V_{n-1}} Moreover,
\w{\phi_{i}} factors through some finite coproduct
\w{\coprod_{j=1}^{k_{i}}\,H\sp{\ast}(\KR{n_{j}};R)} of free monogenic
\TRal[s] in \w[,]{V\sb{n-1}} each of which is in turn indexed by a map
\w[,]{\psi_{i,j}:H\sp{\ast}(\KR{n_{j}};R)\to V_{n-2}} and so on. Thus the
original summand \w{H\sp{\ast}(\KR{n_{i}};R)} of \w{\oV{n}} is ultimately
indexed by a composable sequence of \w{n+1} maps between finitely generated
free \TRal[s,] except for the very last, which lands in
\w[.]{H\sp{\ast}(\bY;R)=\pi_{0}\fX}

Moreover, the composite
\w{(\coprod_{j=1}\sp{k\sb{i}}\,\psi_{i,j})\circ\phi_{i}} vanishes, since
\w{d_{0}\circ\odz{n}} in a CW object. This means that when we realize \w{\Vd}
by the cosimplicial space \w[,]{\Wu} the corresponding composite will be
null-homotopic. This is the source of the map
\w{\etk{n-2}:\uWn{n-2}{n-1}\to\Omega\uW{n}} (in step \w{k=n-2} of the
descending induction in the proof), which is in fact just the Toda bracket,
or secondary cohomology operation, associated to this nullhomotopy (see
\cite[\S 4.1]{HarpSC}).  The $0$-th face map into the higher loops
\w{\Omega\sp{n-k-1}\uW{n}} are analogously associated to higher order
cohomology operations, corresponding to the initial segments of the above
composable sequence of homotopy classes of maps indexing each summand
\w[.]{H\sp{\ast}(\KR{n_{i}};R)}

See \cite{BJTurnHA} for a detailed discussion of higher homotopy operations
and simplicial spaces in the dual setting.
\end{mysubsection}

%
%
\sect{Recognizing and realizing mapping algebras}
\label{crma}

We are now in a position to address the fundamental question of
\emph{recognizing} mapping spaces of the form \w[,]{X=\mapa(\bY,\bA)}
given $\bA$.  Unfortunately, we are able to give a satisfactory
answer only when \w{\bA=\KR{n}} for \w{R=\QQ} or \w[.]{\Fp}

First, we note the following elementary fact:

\begin{lemma}\label{lcocycles}
If $R$ is a field, and \w{G\to\Bu} is a coaugmented cosimplicial $R$-module
such that the dual augmented simplicial $R$-module \w{\Hm{(\Bu)}\to\Hm{G}}
is acyclic, then \w{G\to\Bu} is acyclic, too.
\end{lemma}

Here \w{\Hm{W}:=\Hom\sb{R}(W,R)} is the $R$-dual of an $R$-module $W$.

\begin{proof}
Write \w{\Ad} for \w[.]{\Hm{(\Bu)}} By Definition \ref{dmco}, the
$n$-th Moore chains object \w{C\sb{n}\Ad} of \w{\Ad\to\Hm{G}} is obtained
by applying \w{\Hom\sb{R}(-,R)} to the $n$-th Moore cochains object
\w{C\sp{n}\Bu} of \w[,]{G\to\Bu} since the latter is defined by the colimit
\wref[,]{eqmoorecc} and the former by the corresponding limit \wref[.]{eqmoor}

If we think of the cochain complex \w{C\sp{\ast}\Bu} as a negatively-indexed
chain complex \w[,]{E\sb{\ast}} and thus of \w{C\sb{\ast}\Ad} as
the cochain complex \w[,]{\Hom\sb{R}(E\sb{\ast},R)} we see that
$$
0~=~\pi\sb{i}\Ad~=~H\sb{i}(C\sb{\ast}\Ad)~=~H\sp{-i}(E\sb{\ast},R)~\xra{\cong}~
\Hom\sb{R}(H\sb{-i}(E\sb{\ast}),R)~,
$$
\noindent by the Universal Coefficient Theorem
(cf.\ \cite[Theorem 3.6.5]{WeibHA}). Since the homology groups
\w{H\sb{-i}(E\sb{\ast})\cong\pi\sp{i}\Bu} are free $R$-modules, they must
all vanish.
\end{proof}

\begin{defn}\label{dftsc}
For any ring $R$, and \w{\TsR} as in \S \ref{egeth}, a \TRal $\Lambda$ is
\begin{enumerate}
\renewcommand{\labelenumi}{(\alph{enumi})~}
\item \emph{$k$-connected} if \w{\Lambda\lin{\KR{i}}=0} for \w[.]{0\leq i\leq k}
\item \emph{Finite type} if \w{\Lambda\lin{\KR{n}}} is a finitely generated
$R$-module for each \w[.]{n\geq 0}
\item \emph{Finite} if it is finite type, and there is an $N$ such that
\w{\Lambda\lin{\KR{n}}} vanishes for \w[.]{n\geq N}
\item \emph{Allowable} if there is a partially ordered set $J$, with the
under category \w{j/J} finite for each \w[,]{j\in J} and a diagram of
finite \TRal[s] \w{\uL:J\to\TRA} such that
\begin{myeq}\label{eqallowlim}
\Lambda~=~\lim\sb{j\in J}\,\uL\sb{j}~
\end{myeq}
\noindent in \w[.]{\TRA}
\end{enumerate}

We say that an \Rma $\fX$ is \emph{simply-connected}, \emph{finite type},
\emph{finite}, or \emph{allowable} if the \TRal \w{\pi\sb{0}\fX} is such.
\end{defn}

\begin{example}\label{egallow}
Any finite-type \TRal $\Lambda$ is allowable, with \wref{eqallowlim} given
by the directed system of finite truncations of $\Lambda$.
\end{example}

\begin{remark}\label{rallow}
If $R$ is a field with \w[,]{|R|\leq\aleph\sb{0}} such as \w{\Fq}
or $\QQ$, and \w[,]{\dim\sb{R}V=\aleph\sb{0}} then
\w[,]{V\cong\bigoplus\sb{i=1}\sp{\infty}\,R} so
\w{\Hm{V}\cong\prod\sb{i=1}\sp{\infty}\,R} and thus
\w[.]{|\Hm{V}|=|R|\sp{\aleph\sb{0}}>|R|\cdot\aleph\sb{0}} Therefore,
no $R$-module $W$ of dimension \w{\aleph\sb{0}} can be isomorphic to \w{\Hm{V}}
for any $R$-module $V$. Thus a \TRal $\Lambda$ for which some
\w{\Lambda\lin{\KR{n}}} is \ww{\aleph\sb{0}}-dimensional is not allowable.
Similar phenomena occur for other fields and other cardinals.
\end{remark}

\begin{lemma}\label{lallow}
Let $R$ be a field.
\begin{enumerate}
\renewcommand{\labelenumi}{(\alph{enumi})~}
\item If $\Lambda$ is an allowable \TRal[,] then there is a graded $R$-module
\w{M\sb{\ast}} with \w{\Lambda\lin{\KR{i}}\cong\Hm{M\sb{i}}}
(as $R$-modules) for each \w[.]{i\geq 0}
\item Any realizable \Rma \w{\fX=\fMA\bY} is allowable, and
\w[.]{M\sb{\ast}=H\sb{\ast}(\bY;R)}
\end{enumerate}
\end{lemma}

\begin{proof}
\noindent (a)~~Since (co)limits in functor categories are defined object-wise,
this follows from the fact that for a finite dimensional $R$-module $W$,
\w{\Hm{W}} has a natural isomorphism \w[,]{\Hm{(\Hm{W})}\cong W} and thus
\begin{myeq}\label{eqlimcolim}
\begin{split}
\Lambda\lin{\KR{i}}~=~(\lim\sb{j\in J}\,\uL\sb{j})\lin{\KR{i}}~\cong&~
\lim\sb{j\in J}\,\Hm{(\Hm{U\uL\sb{j}\lin{\KR{i}}})}\\
\cong&~\Hm{(\colim\sb{j\in J}\,\Hm{U\uL\sb{j}\lin{\KR{i}}})}~,
\end{split}
\end{myeq}
\noindent where \w{U:\TRA\to\gr\Mod{R}} is the forgetful functor to graded
$R$-modules, which creates all limits in \w{\TRA} since it is a right
adjoint\vsm.

\noindent (b)~~For any \w{\bY\in\Sa} there is a natural weak equivalence

\begin{myeq}\label{eqallowcoh}
H\sp{i}(\bY;R)~\xra{\cong}~\lim\sb{j\in J}~H\sp{i}(\bY\sb{j};R)\hsp
\text{for all}\hsm i\geq 0~,
\end{myeq}
\noindent where \w{(\bY\sb{j})\sb{j\in J}} is the directed system of finite
subcomplexes of $\bY$  (see \cite[VIII (F)]{ESteeF} and \cite[\S 2]{HMeiC}),
and \w{\pi\sb{0}\fX\lin{\KR{i}}=H\sp{i}(\bY,R)\cong\Hom\sb{R}(H\sb{i}\bY,R)}
by the Universal Coefficient Theorem.
\end{proof}

In particular, any free \TRal is realizable, and thus allowable.

\begin{prop}\label{pdualres}
Let \w{R=\Fp} or a field of characteristic $0$, let $\Lambda$ be
an allowable \TRal with \w{\vare:\Vd\to\Lambda} a free simplicial
\TRal resolution, and let \w{M\sb{\ast}} and \w{N\sb{\ast}} be the graded
$R$-modules of Lemma \ref{lallow} for $\Lambda$ and \w[,]{V\sb{0}} respectively.
Then there is a graded $R$-linear map \w{\psi\sb{\ast}:M\sb{\ast}\to N\sb{\ast}}
with \w{\vare\lin{\KR{i}}=\Hm{\psi\sb{i}}} for each \w[.]{i\geq 0}
\end{prop}

\begin{proof}
Since \w{V\sb{0}} is a free \TRal[,]  it is of the form
\w{V\sb{0}=H\sp{\ast}(\bW\sp{0};R)} (for
\w[),]{\bW\sp{0}\simeq\prod\sb{\alpha\in A}\,\KR{n\sb{\alpha}}} so it is allowable.
By \wref{eqallowlim} we know that $\vare$ is determined by a compatible
system of \TRal[-maps] \w[.]{\vare\sb{j}:V\sb{0}\to\Lambda\sb{j}} Moreover,
\w{V\sb{0}\cong\coprod\sb{\alpha\in A}\,H\sp{\ast}(\KR{n\sb{\alpha}};R)} is
a coproduct (over $A$) of monogenic free \TRal[s.] Thus
\w{\vare\sb{j}:V\sb{0}\to\Lambda\sb{j}} is completely determined by choices
of maps of \TRal[s]
\begin{myeq}\label{eqdualmaps}
\vare\sb{\alpha,j}:H\sp{\ast}(\KR{n\sb{\alpha}};R)~\to~\Lambda\sb{j}~,
\end{myeq}
\noindent  with compatibility requirements only with respect to the various
\w[.]{j\in J}

Since both source and target in \wref{eqdualmaps} are finite \TRal[s,] such
maps are completely determined by their duals, which are maps of $R$-coalgebras.

When \w[,]{\chr(R)=0} by the Milnor-Moore Theorem (cf.\ \cite[App.]{MMoorA})
we have:
\begin{myeq}\label{eqinfprod}
H\sb{\ast}(\bW\sp{0};R)~=~
H\sb{\ast}(\prod\sb{\alpha\in A}\,\KR{n\sb{\alpha}};R)~\cong~
\bigotimes\sb{\alpha\in A}~H\sb{\ast}(\,\KR{n\sb{\alpha}};R)~,
\end{myeq}
\noindent where the right-hand side is the product in the category of
cocommutative coalgebras (cf.\ \cite[1.1b]{GoeHH}), since
\w{H\sb{\ast}(\bW\sp{0};R)} is then the cofree cocommutative coalgebra
\w{V(G\sb{\ast})} on the graded $R$-module
\w[,]{G\sb{\ast}:=\pi\sb{\ast}\bW\sp{0}} and the functor $V$ preserves
products since it is right adjoint to the forgetful functor.

When \w{R=\Fp} the duals
\w{\hvare\sb{\alpha,j}:\Hm{\Lambda\sb{j}}\to H\sb{\ast}(\KR{n\sb{\alpha}};R)}
of \wref{eqdualmaps} are maps of \ww{\Fp}-coalgebras over the mod $p$ Steenrod
algebra. In this case, \wref{eqinfprod} still holds by \cite[\S 4.4]{BousHS},
where the right hand side is now the product in the category of
\ww{\Fp}-coalgebras over the Steenrod algebra. Therefore, all the
maps \w{\hvare\sb{\alpha,j}} define a unique map
\w{\hvare\sb{j}:\Hm{\Lambda\sb{j}}\to H\sb{\ast}(\bW\sp{0};R)} as required.
\end{proof}

As noted in \S \ref{realma}, the homotopy type of an $R$-GEM $X$
alone does not allow us to determine whether it is of the form
\w{X\simeq\mapa(\bY,\KR{n})} for some space \w[,]{\bY\in\Sa} and
recover $\bY$ uniquely \wh for this we need an \Rma structure $\fX$ on $X$
(cf.\ \S \ref{damastruc}). Thus we are naturally led to the following\vsn:

\noindent\textbf{Question:}\ Given an \Rma $\fX$, is there a space $\bY$ with
\w[,]{\fX\simeq\fMR\bY} and if so, is $\bY$ uniquely determined\vsn?

The answer is given by the following:

\begin{thm}\label{treal}
When $R$ is $\QQ$ or \w[,]{\Fp} any simply-connected allowable \Rma $\fX$ is
weakly equivalent to \w{\fMR\bY} (cf.\ \S \ref{saos}) for a simply-connected
$R$-complete space \w[,]{\bY\in\Sa} unique up to weak equivalence.
\end{thm}

If we only assume that \w{\bY\in\Sa} realizes $\fX$, it is unique up
to $R$-equivalence.

\begin{proof}
Let \w{\Vd\to\Lambda} be any a free simplicial resolution of the
allowable \TRal $\Lambda$. It is readily verified that if $\Lambda$ is
simply-connected, this resolution may be chosen so that each \w{V\sb{k}} is
simply-connected, since \w{\tilde{H}\sp{i}(K(R,n);R)=0} for
\w[.]{0\leq i<n}

By Theorem \ref{tres}, \w{\Vd} may be realized by a cosimplicial
space \w[,]{\Wu} with each \w{\bW\sp{n}} a simply-connected $R$-GEM,
and the corresponding simplicial \Rma \w{\fVd:=\fMR\Wu} is augmented to $\fX$.
We may assume that \w{\Wu} is Reedy fibrant.
Since \w{\Vd\to\Lambda} is acyclic, for each \w{\KR{i}\in\TR} we have:
\begin{myeq}\label{eqhtpycoh}
\pi\sb{n}\Vd\lin{\KR{i}}~\cong~\pi\sb{n}\Hu{i-n}{\Wu}{R}\cong~\begin{cases}
\Lambda\lin{\KR{i}}&\hsm \text{if}\hsm n=0\\
0&\hsm \text{otherwise.}
\end{cases}
\end{myeq}
\noindent Therefore, by Lemma \ref{lcocycles} the cosimplicial graded
$R$-module \w{H\sb{\ast}(\Wu;R)} satisfies:
\begin{myeq}\label{eqhtpyhmlgy}
\pi\sp{n}H\sb{i-n}(\Wu;R)\cong~\begin{cases}
M\sb{i}&\hsm \text{if}\hsm n=0\\
0&\hsm \text{otherwise.}
\end{cases}
\end{myeq}
\noindent for \w{M\sb{\ast}} as in Lemma \ref{lallow}.

Therefore, because \w{R=\QQ} or \w[,]{\Fp} the homology spectral sequence
for the simplicial space \w{\Wu} (cf.\ \cite{AndG,RecS,BousHS}), with:
\begin{myeq}\label{eqhmlgyss}
E\sp{2}\sb{m,t}~=~\pi\sp{m}H\sb{t}(\Wu;R)~\Longrightarrow~H\sb{t-m}(\Tot\Wu;R)~,
\end{myeq}
\noindent satisfies the hypotheses of \cite[Theorem 3.4]{BousHS}, so it
converges strongly, and \w{\bY:=\Tot\Wu} is simply-connected.
Moreover, \w[,]{H\sb{\ast}(\bY;R)\cong\Hm{\Lambda}} since \wref{eqhmlgyss}
collapses at the \ww{E\sp{2}}-term, so as in Theorem \ref{tresm} (compare
Proposition \ref{pcommon}), \w{\bY\to\Wu} is a weak $\G$-resolution for the
class $\G$ of all simplicial $R$-modules. This implies that \w{\LG\bY\simeq\bY} by
\cite[Theorem 6.5]{BousC} \wh that is, $\bY$ is $R$-complete
(cf.\ \cite[I, \S 5]{BKanH} and \cite[\S 7.7]{BousC}).

Note that the natural map \w{\cu{\bY}\to\Wu} induces a weak equivalence of
simplicial \Rma[s] \w{\fVd\to\cd{\fMR\bY}} (cf. \S \ref{stma}), where
\w[.]{\fVd:=\fMR\Wu} On the other hand, by Theorem \ref{tres} we also
have a weak equivalence \w[.]{\fVd\to\cd{\fX}}

For any simplicial \Rma \w[,]{\fWd} we can define its \emph{realization}
\w{\fZ:=\|\fWd\|:\bT\to\Sa} by letting \w{\fZ\lin{\bB}} denote the diagonal
of the bisimplicial group \w[,]{\fWd\lin{\bB}} for any \w{\bB\in\bT}. Since
\w{P\fZ\lin{\bB}=\fZ\lin{P\bB}\to\fZ\lin{\bB}} is a fibration by
\cite[Theorem 6.2]{AndF}, and the diagonal preserves products and cofibrations,
we see that $\fZ$ satisfies the three conditions of Definition \ref{dma},
so it is an \Rma[.]

Moreover, by \wref{eqhtpyhmlgy} the Quillen spectral sequence for the
bisimplicial group \w{\fWd\lin{\bB}} collapses for any
\w[,]{\bB\in\TR} and thus the natural maps of \Rma[s]
\mydiagram[\label{eqsimpaug}]{
\fMR\bY && \fZ=\|\fWd\| \ar[ll]\ar[rr] && \fX
}
\noindent induced by the simplicial augmentations
\w{\fMR\bY\leftarrow\fWd\to\fX} are both weak equivalences. Thus we see that
$\bY$ indeed realizes $\fX$, up to weak equivalence of \Rma[s.]
The uniqueness up to weak equivalence (for $R$-complete $\bY$) follows from
\cite[I, Lemma 5.5]{BKanH}.
\end{proof}

\end{document}